\documentclass[finite]{siamltex}
\usepackage{amssymb}
\usepackage{amsfonts,amsmath,amscd}
\usepackage{latexsym}
\usepackage{euscript}
\usepackage[T2A]{fontenc}
\usepackage[cp1251]{inputenc}
\usepackage[english]{babel}
\usepackage{multicol}
\usepackage{graphicx}

\title{{STEPWISE  SYNTHESIS OF CONSTRAINED CONTROLS FOR SINGLE INPUT NONLINEAR SYSTEMS OF SPECIAL FORM}
\thanks{The work was partially supported  by Polish Ministry of Science and High Education grant
\mbox{N N514 238438.}}}

\author{Korobov~V.~I., Sklyar~K.~V.\thanks{Institute of Mathematics, Szczecin University,
        Wielkopolska str. 15, Szczecin 70-451,  Poland ({\tt$\{$korobow,sklar$\}$@univ.szczecin.pl}),}
        \and Skoryk~V.~O.\thanks{V.~N.~Karazin Kharkov National University,
         sq. Svobody, 4, 61077, Kharkov, Ukraine ({\tt
         $\{$vkorobov,skoryk$\}$@univer.kharkov.ua}).}}

\usepackage{subfigure}
\usepackage{multicol}
\usepackage{float}

\begin{document}
\newcommand{\ad}{{\rm ad }}
\maketitle

\begin{abstract} The controllability problem for  nonlinear control systems with  one-dimensional
control of the form  $ dx/dt=a(x)+B(x)\beta(x,u)$ is considered,
where $a(x)$ is an $n$-dimensional vector function, $B(x)$ is an
$(n\times m)$-matrix,  and $\beta(x,u)$ is an $m$-dimensional
vector function.  Under certain conditions we reduce such system
to a system    consisting of  $m$ subsystems; in each subsystem
all equations are linear  except of the last one. We use the
controllability function method  to give   sufficient conditions
for controllability of the considered system. We propose an
approach for construction of controls  which transfer an arbitrary
initial point to the rest point in a certain finite time. Each
such control is constructed as a concatenation of a finite number
of positional controls (we call it a stepwise synthesis control).
On each step of our approach  we choose a new synthesis control.
Our approach essentially uses nonlinearity of a system with
respect to a control.  The obtained results are illustrated by
examples. In particular, the problem of the complete stoppage of a
two-link pendulum is solved. We also introduce  the class of
nonlinear systems which is called the class of staircase systems
that provides the applicability  of our approach.
\end{abstract}



\pagestyle{myheadings} \thispagestyle{plain}
\markboth{KOROBOV~V.~I., SKLYAR~K.~V. AND SKORYK~V.~O.}{{STEPWISE
SYNTHESIS OF CONSTRAINED CONTROLS}}

\section{Introduction} \label{skorik-skorik_section_1} Systems with  controls   appearing  linearly  are
most close to linear systems. Such  systems  are well studied and
various methods  are  developed, namely,  differential-geometric
methods, algebraic methods, and those  commonly used for  linear
systems. In particular, the important role is played by the
feedback linearization method.

In this paper we consider systems for which just  non-linearity
with respect to a control allows to solve the controllability
problem. Namely, we consider a class of systems which are
equivalent to systems of differential equations with one
dimensional control
\begin{equation}\label{skorik-intr_1}
y_i^{(n_i)}=H_i(y_1,\ldots, y_1^{(n_1-1)},\ldots, y_m,\ldots,
y_m^{(n_m-1)},u),\quad i=1,\ldots, m,\quad u\in{\Bbb R},
\end{equation} where $y_i^{(s)}$ means the derivative of order $s$
and the functions $H_1,$ $\ldots,$ $H_m$ are non-linear with
respect to $u.$

The basic idea of our approach  consists in the following. We
solve the problem of controllability to  a rest point of the
system (\ref{skorik-intr_1}) step by step. On the first step we
construct a positional control which depends  on  all state
variables, i.e. a control of the form
$$u=u_1\left(y_1,\ldots, y_1^{(n_1-1)},\ldots, y_m,\ldots,
y_m^{(n_m-1)}\right),$$  transferring  state coordinates of the
first equation to the rest point in certain finite time $T_1,$
i.e. $y_1(T_1)=\ldots=y_1^{(n_1-1)}(T_1)=0.$ On the second step we
construct a positional control
$$u=u_2\left(y_2,\ldots, y_2^{(n_2-1)},\ldots, y_m,\ldots,
y_m^{(n_m-1)}\right)$$ which transfers state variables of the
second equation to the rest point in certain finite time
$T_2-T_1,$ i.e. $y_2(T_2)=\ldots=y_2^{(n_2-1)}(T_2)=0,$  and keeps
coordinates  $y_1,$ $\ldots,$ $y_1^{(n_1-1)}$ at  the rest point,
i.e. $y_1(t)=\ldots=y_1^{(n_1-1)}(t)=0$ for $T_1\le t\le T_2.$
This can be done  only in the case when  $H_1$ depends on a
control non-linearly.

Analogously, on the  $i$-th  step we construct a positional
control
$$u=u_i\left(y_i,\ldots, y_i^{(n_i-1)},\ldots, y_m,\ldots,
y_m^{(n_m-1)}\right)$$ which transfers state variables of the
$i$-th equation to the rest point in certain finite time
$T_i-T_{i-1}$ and keeps coordinates  $y_1,$ $\ldots,$
$y_1^{(n_1-1)},$ $\ldots,$ $y_{i-1},$ $\ldots,$
$y_{i-1}^{(n_{i-1}-1)}$ at the rest point, i.e.
$y_1(t)=\ldots=y_1^{(n_1-1)}(t)=\ldots=y_{i-1}(t)=\ldots=y_{i-1}^{(n_{i-1}-1)}(t)=
0$ as $T_{i-1}\le t\le T_i,$ and so on. After $m$ such steps we
obtain the control of the form
$$\begin{array}{l}
u(y_1,\ldots, y_1^{(n_1-1)},\ldots, y_m,\ldots,
y_m^{(n_m-1)};t)=\\
=u_i\left(y_i,\ldots, y_i^{(n_i-1)},\ldots, y_m,\ldots,
y_m^{(n_m-1)}\right), \;\; t\in[T_{i-1},T_i], \; i=1,...,m\;\;
(T_0=0). \end{array} $$  This control transfers the initial point
to the rest point in the time $T=T_m.$

Thus, on each step we choose a new positional control solving the
positional synthesis problem. As a result of our approach, we
construct a programming control which is a concatenation of a
finite number of the positional controls. {\it We call   it a
''stepwise synthesis control''} which  transfers an arbitrary
initial point to the rest point in a certain finite time $T.$

Let us explain our construction by the following  example.
Consider the system
\begin{equation}\label{skorik-r1_pr0_p1}
 \dot y_1=\sin u,\quad  \dot y_2 = u\cos 2 u.  \end{equation}
This system has the form (\ref{skorik-intr_1}), where $n_1=n_2=1.$
Note that this system is not controllable with respect to the
first approximation.  Suppose $(y_{10}, y_{20})$ is an arbitrary
point. On the first step we choose  the control
$u_1(y_1,y_2)=-(\pi/2){\rm sign}\, y_1.$ This control transfers
the initial  point  to the point $y(T_1)=(y_1(T_1),
y_2(T_1))=(0,\pi y_{10}/2+y_{20})$ in the time $T_1=|y_{10}|.$ On
the second step we choose a control $u_2(y_2)$ such that
$y_1(t)=0$ for $t\ge T_1$ and $y_2(T_2)=0$ for a certain finite
$T_2\ge T_1.$ This means that the corresponding trajectory
$y(t)=(y_1(t),y_2(t))$  of the system (\ref{skorik-r1_pr0_p1})
belongs to the subspace $\{(y_1,y_2): y_1=0\}$ for $t\in
[T_1,T_2].$ This can be done by the control $u_2(y_2)=-\pi\, {\rm
sign}\, y_2.$ This control transfers the point  $y(T_1)$ to the
origin in the time $(T_2-T_1)=|y_{10}/2+y_{20}/\pi|.$ Therefore,
the  point $(y_{10}, y_{20})$ is transferred to the origin by the
control
$$u(y_1,y_2;t)=\left\{\begin{array} {cll}
-(\pi/2){\rm sign}\,y_1& \mbox{for}& 0\le t \le |y_{10}|,\\
-\pi\, {\rm sign}\, y_2 & \mbox{for}&  |y_{10}|<t\le
|y_{10}|+|y_{10}/2+y_{20}/\pi|,
\end{array}\right.$$
  along the trajectory of the system (\ref{skorik-r1_pr0_p1}) in
the time $T=|y_{10}|+|y_{10}/2+y_{20}/\pi|.$ Thus, we have a
stepwise synthesis,  i.e.  on the segment $[0,T_1]$ we choose the
position control $-(\pi/2){\rm sign}\,y_1,$ and on the segment
$[T_1, T_2]$ we choose another positional control $-\pi\, {\rm
sign}\, y_2.$ Note that the times $T_1$ and $T_2$ are not given in
advance but  depend on the initial point $(y_{10},y_{20}).$

In the paper  we introduce a new class of nonlinear single input
systems
\begin{equation}\label{skorik-r1_f2}
  \dot x =a(x)+\sum\limits_{i=1}^mb_i(x) \beta_i(x,u)\equiv a(x)+B(x) \beta(x,u),\quad x\in{\Bbb  R}^n, \quad
u\in {\Bbb  R},  \end{equation} where $ B(x)$ is a
$(n{\times}m)$-matrix $(2\le m\le n)$ with columns $b_1(x),$
$\ldots,$ $b_m(x),$ $\beta(x,u)$ is a $m$-dimensional
vector-function with components $\beta_1(x,u),$ $\ldots,$
$\beta_m(x,u),$ and $u$ is a one-dimen\-sional control. On the
first glance the system (\ref{skorik-r1_f2})  looks like an affine
control system of the form  \begin{equation}\dot
x=a(x)+B(x)u,\quad x\in{\Bbb R}^n,\; u\in {\Bbb
R}^m.\label{skorik-affine_system}\end{equation} However, let us
emphasize that in the system (\ref{skorik-r1_f2}) the control $u$
is only one-dimen\-sional  and, moreover, the nonlinearity of
$\beta(x,u)$ with respect to $u$ plays the crucial role in our
approach. Though an arbitrary single input nonlinear system $\dot
x=f(x,u)$   can be written in the form (\ref{skorik-r1_f2}) in
different ways, nevertheless, not every form is appropriate  for
the further analysis.

Within our approach we deal with systems of the form
(\ref{skorik-r1_f2}) which can be mapped to  systems of  the form
\begin{equation} \label{skorik-r1_kns}
\dot z=A_0z+B_0H(z,u),\quad z\in{\Bbb R}^n, u\in {\Bbb
R},\end{equation} where  $A_0={\rm diag}(A_{1},\ldots,A_{m})$ is a
constant $(n{\times}n)$-matrix, $B_0=(e_{s_1}, \ldots, e_{s_m})$
is a constant $(n{\times}m)$-matrix ($e_{s_i}$ is the $s_i$-th
unit vector of the space ${\Bbb R}^n,$ $i=1,\ldots, m$), and
$H(z,u)$ is a $m$-dimensional vector function.  The system
(\ref{skorik-r1_kns}) is equivalent to the system
(\ref{skorik-intr_1}). In
Section~\ref{skorik-skorik_mappability_of_nonlinear_systems} we
give conditions of the mappability of the system
(\ref{skorik-r1_f2}) on the system (\ref{skorik-r1_kns}). These
conditions are similar  to the linearizability conditions for
affine systems (\ref{skorik-affine_system}). Notice that changes
of variables can be used to increase the amount of rest points of
considered  systems (see Example~\ref{skorik-example_3}). This is
extremely important for our method.

The problem of linearizability for affine systems is well studied
and the conditions  are well  known
\cite{krener}--\cite{skorik-sklyar_sklyar_ignatovich}. However,
generally these conditions are not easy  for check. Therefore it
is important to find classes of systems  for which these
conditions are automatically satisfied. The first such class of
systems called ''the class of triangular systems'' was introduced
in the paper \cite{korobov}, where the feedback linearization was
given. In the paper \cite{korobov_pavlichkov} global properties of
the triangular systems in the singular case is considered. In the
present  paper we introduce the new class of nonlinear systems
called ''the class of staircase systems'' which are mapped on the
systems (\ref{skorik-r1_kns}) and give the corresponding changes
of variables (Section~\ref{skorik-skorik_section_7}).

In  Section~\ref{skorik-skorik_controllability_on_a_subspace} we
solve the problem of positional synthesis to a subspace for
certain class of nonlinear systems. Our main tool is  the
controllability function method  proposed in \cite{korobov1,
korobov_kniga} for solving  the synthesis problem of admissible
positional constrained control. Later it was developed for
different classes of systems and different statements of the
synthesis problem, for example, for infinite systems
\cite{korobov_sklyar_teor_fun}, for systems in a
finite-dimensional space with constraint on a control
\cite{korobov_sklyar} and its  derivatives \cite{korobov_skoryk}
which called inertial control in \cite[p.292]{pontryagin} and so
on. In
Subsection~\ref{skorik-skorik_application_of_the_controllability_function_method}
 we recall the application of the  controllability function method for linear
 systems  \cite{korobov_sklyar}. The main result is
given in Section~\ref{skorik-skorik_main_result}
(Theorem~\ref{skorik-skorik_theorem_2}). Namely, we give
conditions under
 which the application of the method of stepwise synthesis  gives the
solution of the controllability problem  from an arbitrary point
to the rest point of the system.

The obtained  results are  illustrated  by the  examples in
Section~\ref{skorik-skorik_section_examples}. In
Section~\ref{skorik-calming_of_vibrations_of_a_two-link_pendulum}
the problem of complete stoppage  of a two-link pendulum is
solved.

\section{Mappability of nonlinear systems on nonlinear systems of a special form}
\label{skorik-skorik_mappability_of_nonlinear_systems} We consider
the
 problem  of $0$-controllability for the system (\ref{skorik-r1_f2}).
 Suppose
 $ a(x),$ $b_1(x),$ $\ldots,$ $b_m(x)$ are a $n$ times continuously differentiable
vector functions,  $\beta_1(x,u),$ $\ldots,$ $\beta_m(x,u)$ are
 continuously differentiable scalar functions with respect to $x,$
$u,$  and
\begin{equation}\label{skorik-r1_f3} a(0) = 0, \quad  \beta(0,0) = 0. \end{equation}

In this section  we  give  sufficient conditions under which
 system (\ref{skorik-r1_f2}) is mapped on a system of the form
(\ref{skorik-r1_kns}).

 Below we use the following standard notations:  for a scalar continuously differentiable
function $ \varphi(x) = \varphi(x_1,\ldots,x_n) $, denote by
$L_a\varphi $ the derivative of the function $\varphi(x)$ along
the vector field $a(x)$, i.e. $ L_a\varphi(x) = \varphi_x(x)a(x),
$ where $ \varphi_x(x) =
\left(\varphi_{x_1}(x),\ldots,\varphi_{x_n}(x)\right). $ By
$[a(x),b(x)]$ denote the Lie bracket of the vector fields $a(x)$
and $b(x),$ i.e. $[a(x),b(x)]=b_x(x)a(x)-a_x(x)b(x),$ where
$a_x(x),$ $b_x(x)$ are matrices of the first derivatives of
vector-functions $a(x),$ $b(x).$ Also put $ {\rm ad}_a^0 b(x)
=b(x),$ ${\rm ad}_a^k b(x)= [a(x),{\rm ad}_a^{k-1} b(x)],$ $k\ge
1.$

Suppose  for  system (\ref{skorik-r1_f2})  the  condition
\begin{equation}\label{skorik-r1_rang}
{\rm rang}\,Q(x)=n\;\;\mbox{for all}\;\;x\in {\Bbb R}^n,
\end{equation}  holds, where $
Q(x)=(b_1(x), \ldots, b_m(x),\ldots,{\rm ad}_a^{n-1}
b_1(x),\ldots,{\rm ad}_a^{n-1} b_m(x)).$ By $q_i(x),$
$i=1,\ldots,nm,$  denote the columns of the matrix $Q(x).$
 Without loss of generality assume   ${\rm
rang}\,B(x)=m$ for all $x\in {\Bbb R}^n.$ Moreover, since we are
interested in the global $0$-controllability we require that the
vector fields $a(x),$ $b_1(x),$ $\ldots,$ $b_m(x)$ satisfy the
following regularity   property:   for all $j=1,\ldots, nm$
\begin{equation} \label{skorik-blqvblycb}  {\rm rang}(q_1(x),\ldots,
 q_j(x))=c_j \;\; \mbox{for all} \;\; x\in{\Bbb R}^n,\end{equation}
where $c_j$ are certain constants, $1\le c_j\le n.$

Now we delete all columns of the matrix $Q(x)$ that  linearly
depend on previous ones, i.e. columns such that $q_i(x)\in {\rm
Lin} \{q_1(x),\ldots,q_{i-1}(x)\}.$ It is convenient to examine
the columns $q_i(x),$ $i=1,\ldots,nm$,  one by one from left to
right and take into account  the following remark: if the column
$q_i(x)=\ad_a^kb_j(x)$ is deleted  then all columns of the form
$q_{i+ms}(x)=\ad_a^{k+s}b_j(x)$ for all $s\ge 1$ such that
$i+ms\le nm$ should be deleted  as well.  This algorithm is the
same as the algorithm for  linear controllable  systems with
 a multidimensional control  given in the paper  \cite{korobov1} and is  analogous
to the algorithm given in \cite{uonem, jakubchik_respodek,hunt}
for linearization of affine systems with multidimensional control.

 As a result, we obtain the matrix  consisting of the columns of $Q(x)$ which are  not deleted.
It is convenient to permutate these columns and deal with the
matrix $K(x)$ of the form
\begin{equation} \label{skorik-r1_m_K}
K(x){=}\left( b_1(x),\ldots,{\rm ad}_a^{n_1-1}b_1(x),\ldots,
b_m(x),
 \ldots, {\rm
ad}_a^{n_m-1}b_m(x)\right), \end{equation} where
$n_1+\ldots+n_m=n$   and ${\rm rang}\,K(x)=n$ for all  $x\in {\Bbb
R}^n.$

Our main assumption   is as follows: suppose there exist scalar
functions $\varphi_1(x),$ $\ldots,$ $\varphi_m(x)$, which are no
less than {\it twice continuously differentiable} such that:

 (a)  for each $i=1,\ldots,m$ the conditions
\begin{equation} \label{skorik-r1_f4} \left\{\begin{array}{l}
\left(\varphi_i(x)\right)_x{\rm ad}_a^k b_j(x) =0,\;\; k=
0,\ldots,\min\{n_i-2,n_j-1\},\;\;j=1,\ldots,m,\\
\left(\varphi_i(x)\right)_x{\rm ad}_a^{n_i-1} b_i(x) \ne 0, \;\;
x\in{\Bbb  R}^n,\\\varphi_i(0)=0,
\end{array}\right.\end{equation}
are satisfied;

(b) the   change of variables $z=L(x)\in C^{(2)}({\Bbb R}^n)$ of
the form
   \begin{equation} \label{skorik-r1_f5} z_{s_{i-1}+j} = L_a^{j-1}\varphi_i(x), \quad
   j= 1,\ldots,n_i,  \; i=1,\ldots,m,  \end{equation}
is non-singular, i.e. ${\rm det}\,L_x(x)\ne 0$ for all $x\in {\Bbb
R}^n,$ where  $s_0=0,$ $ s_i=n_1{+}\ldots{+}n_i,$ $i=1,\ldots,r.$

 Notice that the conditions of solvability of the system
(\ref{skorik-r1_f4}) in the class of once  continuously
differentiable functions are well known \cite{hartman}, however,
for our aim this is not sufficient.

 Then  in these
variables the system  (\ref{skorik-r1_f2}) takes the form
\begin{equation} \label{skorik-r1_fs}
 \dot z_{s_{i-1}+1}=z_{s_{i-1}+2},\;
\ldots,\;\dot z_{s_i-1}=z_{s_i},\;
 \dot z_{s_i}=H_i(z,u),
\quad i=1,\ldots,m,\end{equation} where
 \begin{equation} \label{skorik-r1_H_i}
 H_i(z,u)
=L_a^{n_i}\varphi_i\left(L^{-1}(z)\right)+\sum\limits_{k=1}^m\beta_k\left(L^{-1}(z),u\right)
L_{b_k}L_a^{n_i-1}\varphi_i\left(L^{-1}z\right).
\end{equation}  For $i=1,\ldots,m$ denote
$z^i=(z_{s_{i-1}+1},\ldots, z_{s_i})^*$ (the sign
* means the transposition). Then  the system  (\ref{skorik-r1_fs}) can be rewritten
as
\begin{equation} \label{skorik-r1_kns_i}
\dot z^i=A_iz^i+b_{0i}H_i(z,u),\quad z^i\in{\Bbb R}^{n_i},\quad
i=1,\ldots,m,
\end{equation} where
$$A_{i}=\left(\begin{array}{cccccc} 0&1&0&\ldots&0&0\\
\ldots&\ldots&\ldots&\ldots&\ldots&\ldots\\
 0&0&0&\ldots&0&1\\
  0&0&0&\ldots&0&0\\
  \end{array}\right),\quad
  b_{0i}=\left(\begin{array}{c}
  0\\
  \ldots\\0\\1\\\end{array}\right).$$
 The system (\ref{skorik-r1_kns_i}) obviously can be written
in the form (\ref{skorik-r1_kns}).

\section{Controllability to a subspace with respect to a part of variables}
\label{skorik-skorik_controllability_on_a_subspace} In this
section we construct a control which transfers any initial point
to a subspace. This is done by use of the controllability function
method \cite{korobov1, korobov_kniga}.

\subsection{Application of the controllability function method for
linear system}
\label{skorik-skorik_application_of_the_controllability_function_method}
The controllability function method gives a general approach for
solving the problem of synthesis of positional constrained
controls. We briefly recall the main ideas  of this method.
Consider the system $$ \dot x=A_0x+b_0v,\quad x\in{\Bbb
R}^{k_1},\;v\in{\Bbb R},$$ with the constraint on a control of the
form
$|v|\le d,$ where $d>0$ is a  given number and
$$
A_0=\left(\begin{array}{cccccc} 0&1&0&\ldots&0&0\\
\ldots&\ldots&\ldots&\ldots&\ldots&\ldots\\
 0&0&0&\ldots&0&1\\
  0&0&0&\ldots&0&0\\
  \end{array}\right),\quad
  b_{0}=\left(\begin{array}{c}
  0\\
  \ldots\\0\\1\\\end{array}\right).$$
Consider a nonsingular $(k_1{\times}k_1)$-matrix $ N(\Theta) =
\int\limits_0^\Theta\left(1-{\frac{t}{\Theta}}\right)
e^{-A_0t}b_0b_0^*e^{-A_0^*t}dt$ \cite{korobov_sklyar}. Suppose
that the number $a_0$ satisfies the condition
\begin{equation} \label{skorik-uslovie_na_a_0}
0<a_0\le 2 d^2/ (N^{-1}(1)b_0,b_0).\end{equation} Define the
controllability function $\Theta(x)$ at $x\ne 0$ as the unique
positive solution of the equation
\begin{equation} \label{skorik-skoryk-r1_f13}
2a_0\Theta=\left(N^{-1}(\Theta)x,x\right) \end{equation} and put
$\Theta(0)=0.$ Then the function $\Theta(x)$ is continuous and
continuously differentiable  for $x\ne 0.$  Choose a control
$v=v(x)$ in the form
\begin{equation} \label{skorik-skoryk-r1_f14}
v(x) = -\frac{1}{2}\,  b_0^*N^{-1}(\Theta (x))x, \quad  x \ne 0.
\end{equation} It can be shown that this control $v(x)$ satisfies
the  Lipschitz condition in each domain $K(\rho_1,\rho_2)=\{x:
0<\rho_1\le \|x\|\le \rho_2\}$ with a Lipschitz constant
$L_v(\rho_1,\rho_2)$ such that $L_v(\rho_1,\rho_2)\to+\infty$ as
$\rho_1\to 0.$

Put $y(\Theta,x)=D(\Theta)x,$ where $D(\Theta)= {\rm
diag}\,\left(\Theta^{-\frac{2k_1-2j+1}{2}}\right)_{j=1}^{k_1}.$
 Rewrite the control (\ref{skorik-skoryk-r1_f14})  in the form
$v(x) = a y(\Theta(x),x) \Theta^{-\frac{1}{2}}(x),$ where
$a=-\frac{1}{2}b_0^*N^{-1}(1).$ Let us show that  the control
satisfies the given constraint for any $x\in {\Bbb R}^{k_1}.$ To
this aim, for a fixed $\Theta$ let us consider the extremal
problem
$$a y(\Theta,x) \Theta^{-\frac{1}{2}}\to {\rm extr},\quad
(N^{-1}(1)y(\Theta,x),y(\Theta,x))-2a_0 \Theta=0.$$  Using the
Lagrange method we get $ y^0=\frac{1}{2\lambda}
\Theta^{-\frac{1}{2}}N(1)a^*$  for an extremum point $y^0.$
 Since $\frac{1}{2\lambda}=\sqrt{2a_0/(N(1)a^*,a^*)}\; \Theta$
then $ay^0
\Theta^{-\frac{1}{2}}=\pm\sqrt{a_0(N^{-1}(1)b_0,b_0)/2}.$ Hence,
the condition (\ref{skorik-uslovie_na_a_0}) implies that the
control (\ref{skorik-skoryk-r1_f14}) satisfies the constraint $
|v(x)| \le d$ for any $x\in {\Bbb R}^{k_1}.$

 Let us calculate the
derivative of the controllability function by virtue of the system
\begin{equation}  \label{skorik-skoryk-r1_zks}
\dot x=A_{0}x+b_0v(x),\quad x\in{\Bbb R}^{k_1}. \end{equation}
Substituting $\Theta=\Theta (x)$ to (\ref{skorik-skoryk-r1_f13})
and differentiating  we obtain
\begin{equation}\label{skorik-skoryk-r1_dottheta_1}
\begin{array}{l}
2a_{0}\dot\Theta=-\left(N^{-1}(\Theta)\widetilde
N(\Theta)N^{-1}(\Theta)x,x\right)\dot\Theta+\left((N^{-1}(\Theta)A_{0}+\right.\\[5pt]
\qquad\qquad\;\;\left.+A_{0}^*N^{-1}(\Theta)x,x\right)-
\left(N^{-1}(\Theta)b_0b_0^*N^{-1}(\Theta)x,x\right),\end{array}
\end{equation} where
$\widetilde N(\Theta)=\frac{1}{\Theta^2}\int\limits_0^\Theta t
e^{-A_{0}t}b_0b_0^*e^{-A_{0}^*t}dt.$ Since
$$A_{0}N(\Theta)+N(\Theta)A_{0}^*=-\int\limits_0^\Theta\left(1-\frac{t}{\Theta}\right)
d\left(e^{-A_{0}t}b_0b_0^*e^{-A_{0}^*t}\right)=b_0b_0^*-\widehat
N(\Theta),$$ where $\widehat
N(\Theta)=\frac{1}{\Theta}\int\limits_0^\Theta
e^{-A_{0}t}b_0b_0^*e^{-A_{0}^*t}dt,$ hence,
\begin{equation}\label{skorik-skoryk-r1_dottheta_2}
N^{-1}(\Theta)A_{0}+A_{0}^*N^{-1}(\Theta)=N^{-1}(\Theta)b_0b_0^*N^{-1}(\Theta)-
N^{-1}(\Theta)\widehat N(\Theta)N^{-1}(\Theta).
\end{equation} Then, denoting
$w=N^{-1}(\Theta)x,$  using  (\ref{skorik-skoryk-r1_f13}),
(\ref{skorik-skoryk-r1_dottheta_2}),
(\ref{skorik-skoryk-r1_dottheta_1}), and taking into account  the
form of the  matrices $\widehat N(\Theta),$ $\widetilde
N(\Theta),$ $N(\Theta)$ we get
$$\dot{\Theta}(x)_{\bigl|(\ref{skorik-skoryk-r1_zks})} =
-(\widehat N(\Theta)w,w)\Bigl/\left(\frac{1}{\Theta}(
N(\Theta)w,w)+(\widetilde N(\Theta)w,w)\right)=-1.$$  Thus, the
time of motion $T(x_0)$ from $x_0\in {\Bbb R}^{k_1}$ to  $x_T=0$
equals $\Theta(x_0),$ where $\Theta(x_0)$ is the positive solution
of the equation (\ref{skorik-skoryk-r1_f13}) at $x=x_0$.

\subsection{Controllability to a subspace} Solutions of all
considered systems  are understood in the  sense of
 differential inclusions \cite{filippov}.

At first, we consider the problem of controllability on a subspace
with respect to a part of variables  for  the system
\begin{equation} \label{skorik-skoryk-skhg}
\dot z=\left(\begin{array}{l}
\dot x\\
\dot y\\
\end{array}\right)=
 \left(\begin{array}{c}
A_0x+b_0 h(x,y,u)\\
g(x,y,u)\\
\end{array}\right),\quad z\in {\Bbb R}^k, x\in{\Bbb R}^{k_1}, y\in {\Bbb R}^{k_2},
u\in{\Bbb R}, \end{equation} where $h(z,u)=h(x,y,u)$ is a
continuous scalar function, $g(z,u)=g(x,y,u)$ is a continuous
$k_2$-dimensional vector function which  satisfy the  Lipschitz
condition with respect to $z$ and $u$ in  each domain $\{(z,u):
0<\rho_1\le \|z\|\le \rho_2, |u|\le \rho_3\}.$

Let us fix some number $d>0.$ Choose  $a_0$ satisfying
(\ref{skorik-uslovie_na_a_0}) and define $\Theta(x)$   as the
unique positive solution of the equation
(\ref{skorik-skoryk-r1_f13}) at $x\ne 0$ and put $\Theta (0)=0.$
Denote  $S^+=\{z\in{\Bbb R}^k: b_0^*N^{-1}(\Theta (x))x>0\},$ $
S^-=\{z\in{\Bbb R}^k: b_0^*N^{-1}(\Theta (x))x<0\},$ and $
S=\{z\in{\Bbb R}^k: b_0^*N^{-1}(\Theta (x))x=0\}.$

\begin{lemma}\label{skorik-skorik-lemma_1} Consider the system (\ref{skorik-skoryk-skhg}).
Suppose there exist two functions $u^+(z),$ $u^-(z)$ which satisfy
the Lipschitz condition in each set $K(\rho_1,\rho_2)=\{z:
0<\rho_1\le\|z\|\le \rho_2\}$   and satisfy the inequalities
\begin{equation} \label{skorik-skoryk-r1_usl_na_ui}
h(z,u^+(z))\ge d,\quad h(z,u^-(z))\le -d.\end{equation}

Then the control $u(z)=u(x,y)$ of the form
\begin{equation}  \label{skorik-skoryk-r1_uz}
u(z) = \left\{
\begin{array}{lll}
  u^-(z) & \mbox{if} \ & z\in S^+, \\
  u^+(z)& \mbox{if} \ &   z\in S^-, \\
 u^0(z) \in \left[u^-(z),u^+(z)\right] &  \mbox{if}  &
z\in S,
 \end{array}\right.\end{equation}
 transfers any point
$z_0=(x_0,y_0)$ to a point
 $z_{T}=(0,y_T)$
 along a trajectory of the system (\ref{skorik-skoryk-skhg}) in a certain finite
time $T(x_0)\le \Theta (x_0)$.
\end{lemma}

\begin{proof} Consider the first subsystem  of the system
(\ref{skorik-skoryk-skhg}) with $u=u(x,y).$ We have
\begin{equation} \label{skorik-skoryk-r1_zkns_i}
\dot
x=A_{0}x+b_0v(x)+b_0\Bigl(h\bigl(x,y,u(x,y)\bigr)-v(x)\Bigr),\quad
x\in{\Bbb R}^{k_1}.  \end{equation} Let us show  that  $
\dot{\Theta}(x)_{\bigl|(\ref{skorik-skoryk-r1_zkns_i})} \le -1.$
Substituting $\Theta=\Theta (x)$ to (\ref{skorik-skoryk-r1_f13})
and differentiating by virtue of the system
(\ref{skorik-skoryk-r1_zkns_i}) we obtain $$ \dot
\Theta(x)_{\bigl|(\ref{skorik-skoryk-r1_zkns_i})} = -1 +
\frac{2b_0^*N^{-1}(\Theta)x\Bigl(h\bigl(x,y,u(x,y)\bigr)-v(x)\Bigr)}
{\frac{1}{\Theta}(N^{-1}(\Theta)x,x) + (N^{-1}(\Theta)\widetilde
N(\Theta)N^{-1}(\Theta)x,x)}.$$  Since $|v(x)|\le d,$ the
inequalities  (\ref{skorik-skoryk-r1_usl_na_ui}) imply
inequalities $h(x,y,u^+(x,y))-v(x)\ge 0,$ $h(x,y,u^-(x,y))-v(x)\le
0, $ hence,
$$\frac{b_0^*N^{-1}(\Theta)x\Bigl(h\bigl(x,y,u(x,y)\bigl)-v(x)\Bigr)}
{\frac{1}{\Theta}\bigl(N^{-1}(\Theta)x,x\bigr)
+\bigl(N^{-1}(\Theta)\widetilde
N(\Theta)N^{-1}(\Theta)x,x\bigr)}\le 0,
 $$
what gives     $ \dot \Theta(x)_{|(\ref{skorik-skoryk-skhg})} \le
-1$ for all $z=(x,y)^*\in {\Bbb R}^k$ such that  $x\ne 0.$ This
means that the
 solution of
Cauchy's problem for the system (\ref{skorik-skoryk-skhg}) with
$u=u(z)$ of the form (\ref{skorik-skoryk-r1_uz})  exists on the
interval $[0,T(x_0))$ and finishes at the point $z_T=(0,y_T)$ in
the finite time $T(x_0)\le \Theta(x_0)$ \cite{korobov_kniga}.
\end{proof}

Further, we consider the problem of controllability to a subspace
with respect to a part of variables  for  the system
(\ref{skorik-r1_kns_i}). For any fixed $i\in \{1,\ldots,m\}$
consider the nonsingular $(n_i{\times}n_i)$-matrix $ N_i(\Theta) =
\int\limits_0^\Theta\left(1-{\frac{t}{\Theta}}\right)
e^{-A_{i}t}b_{0i}b_{0i}^*e^{-A_{i}^*t}dt$ and  choose a number
$a_{0i}$ such that $ 0<a_{0i}\le
2d_i^2/(N_i^{-1}(1)b_{0i},b_{0i})$ for a given $d_i>0.$ Introduce
the  controllability function
 $\Theta_i(z^i)$   as  the unique positive solution of the equation
$ 2a_0\Theta=\left(N_i^{-1}(\Theta)z^i,z^i\right)$ at  $z^i\ne
0^i$ and put $\Theta_i(0^i)=0.$ Lemma~\ref{skorik-skorik-lemma_1}
implies the following theorem.

 \begin{theorem}\label{skorik-skorik_theorem_1}
Consider  system
 (\ref{skorik-r1_kns_i}).
Suppose there exist two functions $u^+(z),$ $u^-(z)$ which satisfy
the Lipschitz condition in each set $K(\rho_1,\rho_2)=\{z:
0<\rho_1\le\|z\|\le \rho_2\}$  and  satisfy the inequalities
$$H_i(z,u^+(z))\ge d_i,\quad H(z,u^-(z))\le -d_i.$$

 Then the control  $u(z)$ of the form $$
u(z) = \left\{
\begin{array}{lll}
  u^-(z), & \mbox {if} \ & z\in S^+, \\
  u^+(z),& \mbox{if} \ &   z\in S^-, \\
 u^0(z) \in \left[u^-(z),u^+(z)\right], &  \mbox{if}  &
z\in S,
 \end{array}\right.$$
where $S^+{=}\{z\in{\Bbb R}^n:
b_{0i}^*N_i^{-1}(\Theta_i(z^i))z^i>0\},$ $ S^-{=}\{z\in{\Bbb R}^n:
b_{0i}^*N_i^{-1}(\Theta_i(z^i))z^i<0\},$
 $ S{=}\{z\in{\Bbb
R}^n: b_{0i}^*N_i^{-1}(\Theta_i(z^i))z^i=0\},$ transfers any point
$z_0=(z_0^1,$ $\ldots,$ $z_0^{i-1},z_0^i,z_0^{i+1},$ $\ldots,$
$z_0^m)^*$ to the point
 $z_{T}=(z^1_{T},$ $\ldots,$
$z^{i-1}_{T},0^i,z^{i+1}_{T},$ $\ldots,$ $z^m_{T})^*$ along a
trajectory of the system  (\ref{skorik-r1_kns_i}) in a certain
finite time $T(z_0^i)\le \Theta_i(z_0^i)$.
\end{theorem}\vskip2mm

\begin{corollary}\label{skorik-skorik-corollary_1}
Let us fix  $i\in\{1,\ldots,m\}.$  Suppose there exist two
functions $u_i^+(z^i,\ldots,z^m),$ $u_i^-(z^i,\ldots,z^m)$ which
satisfy the following conditions

(i)  $u_i^\pm(z^i,\ldots,z^m)$ are Lipschitz functions in each set
$$K_i(\rho_1,\rho_2)=\{z=(0^1,\ldots,0^{i-1},z^i,\ldots,z^m):
0<\rho_1\le\|z\|\le \rho_2\};$$

(ii)
$\;H_k(0^1,\ldots,0^{i-1},z^i,\ldots,z^m,u_i^\pm(z^i,\ldots,z^m))=0$
for all $k=1,\ldots,i-1,$

(iii)
$H_i(0^1,\ldots,0^{i-1},z^i,\ldots,z^m,u_i^-(z^i,\ldots,z^m))\le
-\varepsilon_i^-,$ \\
 $\hspace*{10mm}\;\;\;\, H_i(0^1,\ldots,0^{i-1},z^i,\ldots,z^m,u_i^+(z^i,\ldots,z^m))\ge
 \varepsilon_i^+$

\noindent  for some numbers $\varepsilon_i^\pm >0.$

Put $d_i=\min\{\varepsilon_i^{-},\varepsilon_i^{+}\}$ and
denote
$$\begin{array}{l}
S_i^+=\{z=(0^1,\ldots,0^{i-1},z^i,\ldots,z^m)^*\in{\Bbb R}^n:
b_{0i}^*N_i^{-1}(\Theta_i(z^i))z^i>0\},\\[10pt]
S_i^-=\{z=(0^1,\ldots,0^{i-1},z^i,\ldots,z^m)^*\in{\Bbb R}^n:
b_{0i}^*N_i^{-1}(\Theta_i(z^i))z^i<0\},\\
\end{array}$$
\begin{equation} \label{skorik-surface_S_i}
S_i=\{z=(0^1,\ldots,0^{i-1},z^i,\ldots,z^m)^*\in{\Bbb R}^n:
b_{0i}^*N_i^{-1}(\Theta_i(z^i))z^i=0\}.\end{equation}

Suppose also that

 (iv) the surface $S_i$ of the form
(\ref{skorik-surface_S_i})  is a
 switching surface of the control $u_i^\pm(z^i,$ $\ldots,$ $z^m)$ or there exists a
 control  $u_i^0(z^i,$ $\ldots,$ $z^m)$  such
that  the corresponding trajectory belongs to the surface $S_i$
and $$H_k(0^1,\ldots,0^{i-1},z^i,\ldots,z^m,
u_i^0(z^i,\ldots,z^m))=0\;\;\mbox{for all}\;\; k=1,\ldots,i-1.$$

Then the control
\begin{equation} \label{skorik-r1_uz}
u_i(z^i,\ldots,z^m) = \left\{
\begin{array}{lll}
  u_i^-(z^i,\ldots,z^m) & \mbox {if} \ & z\in S_i^+, \\
  u_i^+(z^i,\ldots,z^m)& \mbox{if} \ &   z\in S_i^-, \\
 u_i^0(z^i,\ldots,z^m) &  \mbox{if}  & z\in S_i,
 \end{array}\right. \end{equation}
 transfers any point
$z_0=(0^1,\ldots, 0^{i-1},z_0^i,z_0^{i+1},\ldots,z_0^m)^*$ to the
point
 $z_{T}=(0^1,$ $\ldots,$
$0^{i-1},0^i,z^{i+1}_{T},$ $\ldots,$ $z^m_{T})^*$ along a
trajectory of the system  (\ref{skorik-r1_kns_i}) in a certain
finite time $T(z_0^i)\le \Theta_i(z_0^i)$.
\end{corollary}

\begin{proof} We consider the Cauchy problem
 $$\dot z^k=A_kz^k+b_{0k}
H_k(z^k,\ldots,z^m,u_i(z^i,\ldots,z^m)),\quad
 z^k(0) = 0^k,\quad k=1,\ldots,i{-}1,$$
for the subsystem of the system (\ref{skorik-r1_kns_i})  with the
control of the form (\ref{skorik-r1_uz}). The conditions (ii),
(iv) give $z^1(t)\equiv 0^1,$ $\ldots,$ $z^{i-1}(t)\equiv 0^{i-1}$
for $t\ge 0.$   Applying Theorem~\ref{skorik-skorik_theorem_1} for
the point
 $z_0=(0^1,\ldots,0^{i-1},z_0^i,z_0^{i+1},\ldots,z_0^m)^*$ and for the control
$u(z)=u_i(z^i,\ldots, z^m)$ of the form (\ref{skorik-r1_uz}) we
get the statement of the corollary.~\end{proof}

\section{Main result}\label{skorik-skorik_main_result} In this section
we give sufficient conditions of $0$-controllability for  system
(\ref{skorik-r1_f2}) which is mapped  on the system
(\ref{skorik-r1_kns}) by the change of variables
(\ref{skorik-r1_f5}).

\begin{theorem}\label{skorik-skorik_theorem_2}  Consider the system
(\ref{skorik-r1_f2}) and suppose that the conditions
(\ref{skorik-r1_f3})
--- (\ref{skorik-r1_f5}) hold.  Consider the functions (\ref{skorik-r1_H_i}).
Suppose for each $i=1,\ldots,m$ there exist two functions
$u_i^{\pm }(z^i,\ldots,z^m)$
  which satisfy conditions  of
  Corollary~\ref{skorik-skorik-corollary_1}.

Then the system  (\ref{skorik-r1_f2})  is  $0$-controllable from
an arbitrary point $x_0$  in a certain finite time $T$.
\end{theorem}

\begin{proof} As it was shown in Section~\ref{skorik-skorik_mappability_of_nonlinear_systems},
it follows from the conditions (\ref{skorik-r1_f3}) ---
(\ref{skorik-r1_f4}) that the system (\ref{skorik-r1_f2}) is
mapped on the system (\ref{skorik-r1_kns_i}) and the map
(\ref{skorik-r1_f5}) takes  any initial point $x_0$ to the point
 $$ z_0 =
 \left(\varphi_1(x_0),\ldots,L_a^{n_1-1}\varphi_1(x_0),\ldots,
 \varphi_m(x_0),\ldots, L_a^{n_m-1}\varphi_m(x_0)\right).$$
Moreover,   the  point $x=0$ is mapped  to the point $z=0.$

Further we  perform  $m$ steps for $i=1,\ldots,m.$ For $i=1$ under
the suppositions of the theorem the control $ u_1(z^1,\ldots,z^m)$
of the form (\ref{skorik-r1_uz}) exists and satisfies the
conditions of Corollary~\ref{skorik-skorik-corollary_1}. Then
Corollary~\ref{skorik-skorik-corollary_1}  implies that the
control $u_1(z^1,\ldots, z^m)$ transfers the point $z_0$ to the
point $z_{T_1}=(0^1,z^2_{T_1},\ldots,z^m_{T_1})^*$ by virtue of
the system (\ref{skorik-r1_kns}) in some finite  time $T_1\le
\Theta_1(z_0^1).$ Suppose after  $(i-1)$ steps  ($2\le i\le m$) we
have constructed  the control
$$ u(z^k,\ldots,z^m;t) = u_k(z^k,\ldots,z^m) \quad\mbox{as}\quad T_{k-1}\le t<T_k \quad (T_0=0),\quad
k=1,\ldots,i-1,
$$ which  transfers the point  $z_0$ to the point
$z_{T_{i-1}}=(0^1,\ldots,0^{i-1},z^i_{T_{i-1}},$ $\ldots,$
$z^m_{T_{i-1}})^*$ by virtue of  the system (\ref{skorik-r1_kns})
in time $T_{i-1}\le
\sum\limits_{k=1}^{i-1}\Theta_k\bigl(z_{T_{k-1}}^k\bigr).$ Let us
 consider the $i$-th step. Under the suppositions of the theorem
the control  $ u_i(z^i,\ldots,z^m)$ of the form
(\ref{skorik-r1_uz}) exists and satisfies the conditions of
Corollary~\ref{skorik-skorik-corollary_1}.
 Then  Corollary~\ref{skorik-skorik-corollary_1}  implies that the
control $u(z^i,\ldots, z^m)$ transfers the point $z_{T_{i-1}}$ to
the point
$z_{T_i}=(0^1,\ldots,0^i,z^{i+1}_{T_i},\ldots,z^m_{T_i})^*$ by
virtue of the system (\ref{skorik-r1_kns}) in some finite time
$T_i-T_{i-1}\le \Theta_i(z_{T_{i-1}}^i).$

Thus,  after  $m$ steps we obtain that the control
$$ u(z;t) = u_i(z^i,\ldots,z^m) \quad\mbox{as}\quad T_{i-1}\le t<T_i\quad (T_0=0),
\quad i=1,\ldots,m,
$$ where  $u_i(z^i,\ldots,z^m)$ are of the form
(\ref{skorik-r1_uz}),  transfers
 the point $z_0$ to the point $z_T=0$  in some time $T=T_m\le
\sum\limits_{k=1}^m\Theta_k\bigl(z_{T_{k-1}}^k\bigr)$ along the
trajectory of the system
 (\ref{skorik-r1_kns}). This trajectory
 has the form $ z(t) = (z^1(t),\ldots,z^m(t))^*,$ where $z^i(t)=0^i$ as $
T_i\le t\le T$ for $i=1,\ldots,m-1.$

Returning to the system (\ref{skorik-r1_f2})  we find the
trajectory $x(t)$  connecting the point $x_0$ with the point
$x_T=0$ as a solution  of the nonlinear system
$$ \varphi_i(x_1,\ldots, x_n)=z_{s_{i-1}+1}(t),\;\ldots,\;
 L_a^{n_i-1}\varphi_i(x_1,\ldots,x_n)=z_{s_i}(t),\quad i=1,\ldots,m.
  $$
This trajectory also  can be found   by integrating  the system
(\ref{skorik-r1_f2}) with the control $u(t)=u(z(t))$ at  $0\le
t\le T$ with the initial condition  $ x(0)=x_0.$ \qquad\end{proof}

\section{Examples}\label{skorik-skorik_section_examples} In this section we give several examples
illustrating Theorem ~\ref{skorik-skorik_theorem_2}.

\subsection{\hskip-3.1mm}  Consider the  system
\begin{equation} \label{skorik-r1_pr2_sys2}
\dot x_1 = u^3+0.1\sin^2f_1(x_1,x_2,x_3,u), \quad \dot x_2 = u ,
\quad \dot x_3 = f_2(x_2),\quad |u|\le 2,
\end{equation} where
 $f_1$ and $f_2$ are
continuously differentiable functions such that  $ f_1(0,0,0,0) =
0,$ $f_2(0) = 0,$  $\left| f_2'(x_2)\right|\ge \delta>0. $ We note
that the system (\ref{skorik-r1_pr2_sys2}) is not controllable at
the first approximation in a neighborhood of the stationary point
$(x=0, u=0).$  We consider the $0$-controllability problem from
any point  $ x_0 =(x_{10},x_{20},x_{30})^* $ and construct a
control transferring the point $x_0$ to the origin.

The system (\ref{skorik-r1_pr2_sys2}) can be rewritten in the form
(\ref{skorik-r1_f2}) with  $a(x)=(0,0,f_2(x_2))^*,$ $a(0)=0,$
$b_1(x)=(1,0,0)^*,$ $b_2(x)=(0,1,0)^*,$
$\beta_1(x,u)=u^3+0.1\sin^2f_1(x_1,x_2,x_3,u),$ $\beta_2(x,u)=u,$
$m=2.$  The matrix $K(x)$ from  (\ref{skorik-r1_m_K})  has the
form
$$K(x)=(b_1(x), b_2(x),{\rm ad}_a b_2(x))=\left(\begin{array}{ccc}
1&0&0\\
0&1&0\\
0& 0& -f_2'(x_2)\end{array}\right),\quad {\rm rang}\, K(x)=3,\;
x\in{\Bbb R}^3,$$ hence, $n_1=1,$ $n_2=2.$    The conditions
(\ref{skorik-r1_f4}) require that functions $\varphi_1(x),$
$\varphi_2(x)$ satisfy the condition
$$\frac{\partial\varphi_1(x_1,x_2,x_3)}{\partial x_1}\ne 0;\;
\frac{\partial\varphi_2(x_1,x_2,x_3)}{\partial x_1}= 0,\;
\frac{\partial\varphi_2(x_1,x_2,x_3)}{\partial x_2}= 0,\; \frac{
\partial\varphi_2(x_1,x_2,x_3)}{\partial x_3} \ne 0.$$
We  choose  $\varphi_1(x_1,x_2,x_3)= x_1-x_2$ and
$\varphi_2(x_1,x_2,x_3)=x_3.$ Then the  non-singular change of
variables (\ref{skorik-r1_f5}) has the form
\begin{equation}  \label{skorik-r1_pr2_zam}
z_1 =x_1-x_2,\quad z_2 =x_3, \quad z_3 =f_2(x_2).
\end{equation} We get
$x_3=z_2,$ $ x_2=f_2^{-1}(z_3),$ $ x_1=z_1+f_2^{-1}(z_3),$ then
\begin{equation}  \label{skorik-r1_pr2_f2}
\dot z_1 = u^3 -u + 0.1\sin^2 \widetilde
{f_1}(z_1,z_2,z_3,u),\quad \dot z _2 = z_3 ,\quad
 \dot z_3 =\widetilde {f_2}(z_3)u,
\end{equation}
 where $\widetilde  {f_1}(z_1,z_2,z_3,u) =
f_1\Bigl(z_1+f_2^{-1}(z_3),f_2^{-1}(z_3),z_2,u\Bigr),$ $\widetilde
{f_2}(z_3)=f_2'(x_2)_{\bigl| x_2=f_2^{-1}(z_3)}.$ The system
(\ref{skorik-r1_pr2_f2}) has the form (\ref{skorik-r1_fs}), where
$$H_1(z,u)=u^3-u+0.1\sin^2\widetilde {f_1}(z_1,z_2,z_3,u),\quad
H_2(z,u)=\widetilde {f_2}(z_3)u.$$

Now consider  the $0$-controllability problem from the point $
z_0= (z_{10},z_{20},z_{30})^* = \left(x_{10}-x_{20},
x_{30},f_2(x_{20})\right)^*.$

 On the first step  of our approach we  find controls
$u_1^+$ and $u_1^-.$ Notice  that the equation
\begin{equation} \label{skorik-r1_pr2_f3} u^3-u+0.1\sin^2\widetilde {f_1}(z_1,z_2,z_3,u) = v
\end{equation}
 has three real
roots on the  segment $[-2,2]$ for all $v$ such that   $ |v|\le
\frac{2}{3\sqrt{3}} -\frac{1}{10}. $ Put $\varepsilon_1^\pm=0.2$
and choose  controls $u_1^+(z_1,z_2,z_3)$ and $u_1^-(z_1,z_2,z_3)$
as the solutions of the equation (\ref{skorik-r1_pr2_f3}) with
$v=0.2$ and $v=-0.2.$ It can be shown that $u_1^-\in [-1.2, -0.8]$
and $u_1^+\in [0.7, 1.1].$  Then we get
$H_1(z,u^+)=\varepsilon_1^+$ and $H_1(z,u^-)=-\varepsilon_1^-.$
Put
$$u_1(z_1,z_2,z_3)=\left\{\begin{array}{l}
u_1^+(z_1,z_2,z_3)\;\; if \;\;z_{1}<0,\\
u_1^-(z_1,z_2,z_3)\;\; if \;\;z_{1}>0.\\
\end{array}\right.$$
Then this control transfers any initial point $z_0=
(z_{10},z_{20},z_{30})^*$ to the point
$z_{T_1}=\left(0,z_{2T_1},z_{3T_1}\right)^*$ in the time
$T_1=5|z_{10}|.$ Therefore, the trajectory of the system
(\ref{skorik-r1_pr2_f2}) with the control $u=u_1(z_1,z_2,z_3)$
comes to the plane $z_1=0.$

On the second step we choose the control $u_2$ so that the
trajectory of the system (\ref{skorik-r1_pr2_f2}) with the control
$u=u_2$ belongs to the plane $z_1=0.$ This control should satisfy
the equation (\ref{skorik-r1_pr2_f3}) with $v=0$ and $z_1=0.$
This equation has three real  roots for any $z_2,$ $z_3.$
Moreover, it can be shown that two of them   belong to the
segments $[-1.1, -1]$ and  $[0.9,1]$ respectively.

Due to our assumption $|\widetilde f_2(z_3)|\ge \delta>0$ for
$z_3\in {\Bbb R}.$ In the case $\widetilde f_2(z_3)>0$ choose  $
u_2^-(z_2,z_3)\in [-1.1, -1]$ and  $ u_2^+(z_2,z_3)\in[0.9,1].$ In
the case $\widetilde f_2(z_3)<0$ choose  $ u_2^+(z_2,z_3)\in
[-1.1, -1]$ and  $ u_2^-(z_2,z_3)\in[0.9,1].$  Then
$H_1(0,z_2,z_3,u_2^\pm(z_2,z_3))=0.$ Let $\gamma^+$ and $\gamma^-$
be the trajectories of the system
$$
 \dot z_2 = z_3, \quad \dot z_3 =\widetilde {f_2}(z_3)u  $$
 going to the origin and
corresponding to the controls $u=u_2^+(z_2,z_3)$ and
$u=u_2^-(z_2,z_3)$ respectively. The curve $\gamma
=\gamma^+\cup\gamma^-$ breaks  the  plane $z_1=0$ in  two parts.
Put $$u_2(z_2,z_3)\!=\!\left\{\!\!\!\begin{array}{l}
u_2^+(z_2,z_3)\; \mbox{if the point} \; (z_2,z_3)\;\mbox{lies
below the curve}\;
\gamma \;\mbox{or belongs to}\;\gamma^+,\\
u_2^-(z_2,z_3)\; \mbox{if the point} \; (z_2,z_3)\;\mbox{lies
above  the curve}\;
\gamma \;\mbox{or belongs to}\;\gamma^-.\\
\end{array}\right.$$
This control transfers the point
$z_{T_1}=\left(0,z_{2T_1},z_{3T_1}\right)^*$ to the origin in
certain finite  time $(T_2-T_1)$ and the corresponding trajectory
belongs to the plane $z_1=0.$

Returning to the initial variables  we have that the control of
the form
 $$ u(x;t){=}\left\{\!\!\begin{array}{lcl}
   u_1\left(x_1-x_2,x_3,f_2(x_2)\right)& \mbox{as} & 0\le t<T_1, \\
   u_2 \left(x_3,f_2(x_2)\right)& \mbox{as}  & T_1\le t \le T_2, \\
    \end{array}\right. $$
transfers the initial point $ (x_{10},x_{20},x_{30})$ to the
origin by virtue of the initial system~(\ref{skorik-r1_pr2_sys2})
in some finite time $T_2,$ where $T_1$ is defined by
$x_{10}-x_{20}$ and $T_2-T_1$ is defined by $x_{2T_1}$ and
$x_{3T_1}.$

\subsection{\hskip-3.1mm} \label{skorik-example_3}  Consider  the system
\begin{equation} \label{skorik-r1_pr4_f1}
\dot x_1=u,\quad \dot x_2=u^3,
 \quad \ldots,\quad
 \dot x_n=u^{2n-1}, \end{equation}
  with constraints  on a control  of the form
 $u\in \Omega=\{ u: \;|u|\le 1\}.$ The system (\ref{skorik-r1_pr4_f1}) can be written as   $\dot x=\Phi(u),$
 $x\in {\Bbb R}^n$, $\Phi(u)=(u,u^3,\ldots,u^{2n-1})^*.$
 This system is globally $0$-controllable due to the geometrical
 criterion \cite{korobov_25} since the origin $x=0$ belongs to the
 interior of a convex span  of the set $\Phi(\Omega),$ i.e. $0\in {\rm int\; co}
\{\Phi(\Omega)\}.$

The system (\ref{skorik-r1_pr4_f1}) has the form
(\ref{skorik-r1_f2}) with $a(x)=0,$ $b_1(x)=e_n,$ $\ldots,$
$b_n(x)=e_1,$ where $e_i$ is the $i$-th unit vector of the space
${\Bbb R}^n.$       By $P_i(u)$ $(i=1,\ldots,n)$ denote the
polynomial of degree $(2n{-}2i{+}1)$ of the form
$$P_{i}(u) =
    u \prod\limits_{k = 1}^{n-i}\Bigl(u^2 - \frac{k^2}{n^2}\Bigr)=u^{2n-2i+1}+\sum\limits_{k=1}^{n-i}
    c_k^{(i)} u^{2k-1},\; i=1,\ldots,n-1,\quad P_n(u)=u.$$
Notice that all roots of  the polynomial $P_{i+1}(u)$ are the
roots of polynomials $P_1(u),$ $\ldots,$ $P_i(u).$ Put
$$\varphi_i(x_1,\ldots,x_n)=x_{n-i+1}+\sum\limits_{k=1}^{n-i}
    c_k^{(i)}x_k,\;i=1,\ldots,n-1,\;
    \;\varphi_n(x_1,\ldots,x_n)=x_1,$$
then the  conditions (\ref{skorik-r1_f4}) are satisfied. Hence,
the nonsingular  change of variables
$z_i=\varphi_i(x_1,\ldots,x_n),$ $i=1,\ldots,n,$ maps the system
(\ref{skorik-r1_pr4_f1}) to the system
\begin{equation} \label{skorik-r1_pr4_f2}
\dot z_i=P_i(u),\quad i=1,\ldots,n,\quad |u|\le 1, \end{equation}
and an arbitrary point $x_0=(x_{10},\ldots,x_{n0})^*$ is mapped to
a point $z_0=(z_{10},\ldots,z_{n0})^*,$ where $
z_{i0}=x_{n-i+1\,0}+\sum\limits_{k=1}^{n-i}
    c_k^{(i)} x_{k0}$ for $i=1,\ldots,n-1$ and $z_{n0}=x_{10}.$
We choose $u_i(z_i)=-\frac{n-i+1}{n}\;{\rm sign}\,z_i,$
$i=1,\ldots,n,$ and put
\begin{equation} \label{skorik-r1_pr4_upr}
u(z;t)=u_i(z_i)  \quad \mbox{as}\quad T_{i-1}\le t<T_i,\quad
i=1,\ldots,n, \end{equation}
 where $ T_0=0,$
$T_i=T_{i-1}+|z_{iT_{i-1}}/P_i(u_i)|$
($z_{iT_{i-1}}=z_{i}(T_{i-1})$). The control
(\ref{skorik-r1_pr4_upr}) transfers the point  $z_0$ to the origin
by virtue of the system (\ref{skorik-r1_pr4_f2}) in  some finite
time $T=T_n.$ Namely, on the first step the control $u_1=-{\rm
sign}\,z_{10}$ transfers the point $z_0$ to the point
$z_{T_1}=(0,z_{2T_1},\ldots, z_{nT_1})^*,$ where
$z_{kT_1}=P_k(u_1)T_1+z_{k0},$ $k=2,\ldots,n,$
 in the time $T_1=|z_{10}/P_1(u_1)|$ along the trajectory
$z(t)=\left(P_1(u_1)t+z_{10},\ldots, P_n(u_1)t+z_{n0}\right)^*.$
Since $P_1(u_i)=0$ for $i=2,\ldots,n$ then  $z_1(t)=0$ for $t\ge
T_1.$ Further, on the $i$-th step ($i=2,\ldots,n$) the control
 $u=u_i$ transfers the point
 $z_{T_{i-1}}=\left(0,\ldots,0,z_{iT_{i-1}},\ldots,z_{nT_{i-1}}\right)^*$
to the point $z_{T_i}=(0,\ldots,0,z_{i+1T_i},\ldots, z_{nT_i})^*,$
where $z_{kT_i}=P_k(u_i)(T_i-T_{i-1})+z_{kT_{i-1}},$
$k=i+1,\ldots,n,$ in the time
$T_i-T_{i-1}=|z_{iT_{i-1}}/P_i(u_i)|$ along the trajectory
$z(t)=\left(0,\ldots,0,P_i(u_i)(t-T_{i-1})+z_{iT_{i-1}},\ldots,
P_n(u_i)(t-T_{i-1})+z_{nT_{i-1}}\right)^*.$ Since  $P_i(u_k)=0$
for $k=i+1,\ldots,n$ then  $z_i(t)=0$ as  $t\ge T_i.$  Returning
to the initial variables we find $x_1,$ $\ldots,$ $x_n$
successively from the equalities $x_1=z_n,$ $
x_i=z_{n-i+1}-\sum\limits_{k=1}^{i-1}
    c_k^{(n-i+1)} x_k$ for $i=2,\ldots,n.$ Thus, the control $$u(x;t)=-\frac{n{-}i{+}1}{n}{\rm
sign}\left(x_{n{-}i{+}1}{+}\sum\limits_{k=1}^{n-i}
    c_k^{(i)} x_{k }\right)\quad\mbox{as}\quad T_{i-1}\le t<T_i,\quad
i=1,\ldots,n,$$
    satisfies the preassigned constraint   $|u|\le 1$ and
    transfers  an arbitrary point $x_0$ to the origin  in some finite time
 $T(x_0)$ along the trajectory  $x(t)=(x_1(t),\ldots,x_n(t))^*$ of the system
 (\ref{skorik-r1_pr4_f1}), where
 $x_1(t)=z_n(t),$ $x_i(t)=z_{n-i+1}(t)
 -\sum\limits_{k=1}^{i-1}
    c_k^{(n-i+1)} x_k(t),$ $i=2,\ldots,n.$

Notice  that this construction admits an obvious generalization.
Let us  choose numbers $\lambda_1,$ $\ldots,$ $\lambda_{n-1}$ such
that
 $\;0<\lambda_1<\ldots<\lambda_{n-1}< d$ and
introduce the  polynomials
$$P_i(u) =
    u \prod\limits_{k = 1}^{n-i}\left(u^2 - \lambda_k^2\right)=
    u^{2n-2i+1}+\sum\limits_{k=1}^{n-i}
    p_k^{(i)} u^{2k-1},\; i=1,\ldots,n-1,\quad P_n(u)=u.$$
Consider  the nonsingular change of variables
 $z_i=x_{n-i+1}+\sum\limits_{k=1}^{n-i}
  p_k^{(i)} x_k,$  $i=1,\ldots,n{-}1,$  $z_n=x_1$ and choose
  the control  (\ref{skorik-r1_pr4_upr}) with  $u_1(z)=-\alpha\, {\rm sign}\,z_1,$
$\alpha{\in}(0,d]\setminus\{\lambda_1,\ldots,\lambda_{n-1}\},$ and
 $ u_i(z_i,\ldots,z_n)=-\lambda_{n+1-i}\;{\rm sign}\,z_{i}$ for  $i=2,\ldots,n,$ where $
T_i=T_{i-1}+|z_{iT_{i-1}}/P_i(u_i)|,$  $i=1,\ldots,n.$ This
control satisfies the constraint $|u|\le d$ and  transfers an
arbitrary point $z_0$ to the origin in the finite time $T=T_n.$

\section{Calming of vibrations of a two-link pendulum}\label{skorik-calming_of_vibrations_of_a_two-link_pendulum} In this
section we consider the model  of a controllable two-link pendulum
(see fig.~\ref{skorik-ndm-ris1}).

\begin{multicols}{2}
  \begin{figure}[H]
   \vspace*{-3mm}
    \centering
    \includegraphics[scale=0.35]{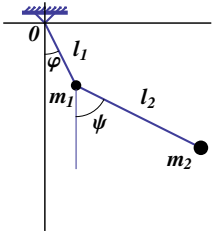}
     \vspace*{-3mm}
    \caption{ Two-link pendulum. }\label{skorik-ndm-ris1}
  \end{figure}
  \begin{figure}[H]
  \vspace*{-3mm}
    \centering
    \includegraphics[scale=0.35]{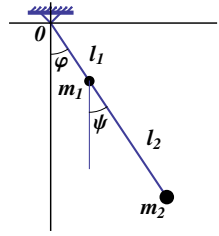}
    \vspace*{-3mm}
    \caption{Two-link pendulum in the time moment  $T_1.$}\label{skorik-ndm-ris2}
  \end{figure}
\end{multicols}

Namely, let  a pendulum have  two links of mass $m_1,$  $m_2$ and
of lengths $l_1,$ $l_2$ respectively. Then the state of the
pendulum is described by  angles $\varphi,$ $\psi$  and  angle
velocities $\dot \varphi,$ $\dot \psi$ ($\varphi$ is the angle
between  the first bar and the  vertical axis; $\psi$ is the angle
between the second bar and the vertical axis). Let $F_1,$ $F_2$ be
forces applied to the first and the second link respectively. Let
 $g$ be  the acceleration of the free fall. We consider the
model of the pendulum with $F_1=\alpha u^3,$ $F_2=u,$ where $u$ is
a control, $\alpha\in(0,(4/27)l_1^2/g^2].$ Suppose the initial
state of the  pendulum $(\varphi_0,\psi_0, \dot \varphi_0,\dot
\psi_0)$ is given.  We construct  a control $u=u(\varphi,\psi,
\dot \varphi,\dot \psi,t)$ which calms the vibrations of the
pendulum, that is   transfers the initial state
$(\varphi_0,\psi_0, \dot \varphi_0,\dot \psi_0)$ to the origin in
some finite time $T,$ i.e.  $\varphi(T)=0,$ $\psi(T)=0,$ $\dot
\varphi(T)=0,$ $\dot \psi(T)=0.$

The control motion of the two-link pendulum
 is described by the equations
\begin{equation} \label{skorik-syst_phi_psi}
\begin{array}{l}
\displaystyle
\ddot\varphi=-\frac{gm_1\sin\varphi+m_2\sin(\varphi{-}\psi)
\left(g\cos\psi+l_1\dot\varphi^2\cos(\varphi{-}\psi)+l_2\dot\psi^2\right)}
{l_1\left(m_1+m_2\sin^2(\varphi{-}\psi)\right)}+\alpha u^3,\\[5pt]
\displaystyle
\ddot\psi=\frac{\sin(\varphi{-}\psi)\left((m_1+m_2)\left(g\cos\varphi+l_1\dot\varphi^2\right)
+m_2l_2\dot\psi^2\cos(\varphi{-}\psi)\right)}
{l_2\left(m_1{+}m_2\sin^2(\varphi{-}\psi)\right)}+u.\end{array}\end{equation}

Put $x_1=\varphi,$ $x_2=\dot \varphi,$ $x_3=\psi,$ $x_4=\dot\psi,$
then we obtain the system
\begin{equation} \label{skorik-dnm_f1_x}
\dot x_1=x_2,\quad \dot x_2=\beta_1(x,u),\quad \dot x_3=x_4,\quad
 \dot x_4=\beta_2(x,u),  \end{equation}
where $$\begin{array}{l} \displaystyle \beta_1(x,u)=
-\frac{gm_1\sin x_1{+}m_2\sin(x_1{-}x_3)\left(g\cos
x_3{+}l_1x_2^2\cos(x_1{-}x_3){+}l_2x_4^2)\right)}{l_1\left(m_1{+}m_2\sin^2(x_1{-}x_3)\right)}+\alpha u^3,\\[3pt]
\displaystyle
\beta_2(x,u)=\frac{\sin(x_1{-}x_3)\left((m_1{+}m_2)(g\cos
x_1{+}l_1x_2^2){+}l_2m_2x_4^2\cos(x_1{-}x_3)\right)}{l_2\left(m_1+m_2\sin^2(x_1{-}x_3)\right)}+u,
\end{array}$$

 The system (\ref{skorik-dnm_f1_x}) can be rewritten in the form
(\ref{skorik-r1_f2}) with  $a(x)=(x_2,0,x_4,0)^*,$ $a(0)=0,$
$b_1(x)=(0,1,0,0)^*,$ $b_2(x)=(0,0,0,1)^*,$  $m=2.$  The matrix
$K(x)$ from  (\ref{skorik-r1_m_K}) has the form
$$K(x)=(b_1(x),{\rm ad}_a b_1(x), b_2(x),{\rm ad}_a b_2(x))=
\left(\begin{array}{rrrr}
0&-1&0&0\\
1&0&0&0\\
0& 0& 0& -1\\
0&0&1&0\end{array}\right)$$ and ${\rm rang}\, K(x)=4,$ $x\in{\Bbb
R}^4,$ hence, $n_1=2,$ $n_2=2.$    The conditions
(\ref{skorik-r1_f4}) imply
$$\begin{array}{l}
\displaystyle \frac{\partial\varphi_1(x_1,x_2,x_3)}{\partial x_2}=
0,\; \frac{\partial\varphi_1(x_1,x_2,x_3)}{\partial x_4}= 0,\;
\frac{\partial\varphi_1(x_1,x_2,x_3)}{\partial x_1}\ne 0,\\
\displaystyle  \frac{\partial\varphi_2(x_1,x_2,x_3)}{\partial
x_2}= 0,\; \frac{\partial\varphi_2(x_1,x_2,x_3)}{\partial x_4}=
0,\; \frac{
\partial\varphi_2(x_1,x_2,x_3)}{\partial x_3} \ne 0,
\end{array}\quad x\in{\Bbb R^4}.$$ Choose  $\varphi_1(x_1,x_2,x_3)=
x_1-x_2$ and $\varphi_2(x_1,x_2,x_3)=x_3.$ Hence, the non-singular
change of variables (\ref{skorik-r1_f5}) has the form
$z_1=x_1-x_3,$ $z_2=x_2-x_4,$ $z_3=x_3,$ $z_4=x_4.$ Then
(\ref{skorik-dnm_f1_x}) is mapped to   the system
\begin{equation}\dot z_1=z_2,\quad\dot z_2=H_1(z_1,z_2,z_3,z_4,u),\quad
\dot z_3=z_4,\quad \dot z_4=H_2(z_1,z_2,z_3,z_4,u),
 \label{skorik-dnm_f1_z}\end{equation}
$$\begin{array}{l}
\displaystyle  H_1(z,u)= -\frac{1}{l_1 l_2 (m_1{+} m_2\sin^2 z_1)}
(l_2 m_1 g \sin (z_1{+}z_3){+}l_2 m_2 \sin z_1 (g \cos z_3{+}l_2
z_4^2{+}\\[5pt]
   \qquad\qquad\quad +
     l_1 (z_2{+}z_4^2)\cos z_1){+}
   l_1\sin z_1 (l_2 m_2 z_4^2\cos z_1{+} (m_1 {+}
         m_2) (g \cos (z_1{+}z_3){+}\\[3pt]
\qquad\qquad\quad + l_1 (z_2{+}z_4)^2))){+}\alpha u^3{-}u,\\[5pt]
 \displaystyle H_2(z,u)=\frac{\sin z_1(l_2m_2z_4^2\cos z_1{+}(m_1{+}m_2)(g\cos(z_1+z_3)+l_1(z_2{+}z_4)^2))}
 {l_2(m_1{+}m_2\sin^2z_1)}+u.
\end{array}$$
Notice that for a fixed $z$ the function $H_1(z,u)$ is a cubic
polynomial  with respect to $u.$ We choose controls $u_1^+(z)$ and
$u_1^-(z)$ as solutions of equations $H_1(z,u)=\varepsilon_1^{+},$
$H_1(z,u)=-\varepsilon_1^{-}$ respectively (we do not give the
explicit  formulas  for these controls since they are too
complicated).

Put $w_1(z_1)=-\sqrt{2\varepsilon_1^{+}z_1}$ for $z_1\ge 0$ and
$w_1(z_1)=\sqrt{-2\varepsilon_1^{-}z_1}$ for $z_1< 0,$ and define
$$ u_1(z)=\left\{\begin{array}{l}
u_1^+(z_1,z_2,z_3,z_4)\;\; \mbox{if} \;\;z_2<w_1(z_1)\;\;\mbox{or}
\;\; z_2=w_1(z_1)\;\;
\mbox{and}\;\;z_1\ge 0,\\
u_1^-(z_1,z_2,z_3,z_4)\;\; \mbox{if}
\;\;z_2>w_1(z_1)\;\;\mbox{or}\;\;z_2=w_1(z_1)\;\;
\mbox{and}\;\;z_1\le 0.\\[3pt]
\end{array}\right. $$
If the initial point $z_0=(z_{10},z_{20},z_{30},z_{40})^*$
satisfies the inequality $z_{20}<w_1(z_{10})$  then the control
$u_1^+(z)$ transfers the point $z_0$ to the point
$z_{T_{11}}=\bigl(z_{1T_{11}},z_{2T_{11}},$
$z_{3T_{11}},z_{4T_{11}}\bigr)^*$ in the time
$T_{11}=\bigl(-z_{20}{+}\sqrt{(z_{20}^2{-}2z_{10}\varepsilon_1^{+})
\varepsilon_1^{-}/(\varepsilon_1^{+}{+}\varepsilon_1^{-})}\bigr)\bigl/\varepsilon_1^{+}$
and further the control $u_1^-(z)$ transfers the point
$z_{T_{11}}$ to the point
$z_{T_1}=\left(0,0,z_{3T_1},z_{4T_1}\right)^*$ in the time
$T_{12}=\sqrt{\frac{z_{20}^2-2z_{10}\varepsilon_1^{+}}{\varepsilon_1^{-}(\varepsilon_1^{+}+
\varepsilon_1^{-})}}.$ Thus, the control $u_1(z)$ transfers the
initial point  $z_0$ to the point $z_{T_1}$ in the time
$T_1=T_{11}+T_{12}=\bigl(-z_{20}{+}\sqrt{(z_{20}^2-2z_{10}\varepsilon_1^{+})
(\varepsilon_1^{+}+\varepsilon_1^{-})/\varepsilon_1^{-}}\bigr)\bigl/\varepsilon_1^{+}.$

If the initial point $z_0$ satisfies the conditions
$z_{20}=w_1(z_{10})$ and $z_{10}\ge 0$ then the control $u_1^-(z)$
transfers the point  $z_0$ to the point
$z_{T_1}=\left(0,0,z_{3T_1},z_{4T_1}\right)^*$  in the time
$T_1=T_{12}.$

Analogously, in the case $z_{20}>w_1(z_{10})$ the control
$u_1^-(z)$ transfers $z_0$ to the point $z_{T_{11}}$ in the time
$T_{11}=\bigl(z_{20}+\sqrt{(z_{20}^2+2z_{10}\varepsilon_1^{-})
\varepsilon_1^{+}/(\varepsilon_1^{+}+\varepsilon_1^{-})}\bigr)\bigl/\varepsilon_1^{-}$
 and further the control
$u_1^+(z)$ transfers the point $z_{T_{11}}$ to the point
$z_{T_1}=\left(0,0,z_{3T_1},z_{4T_1}\right)^*$ in the time
$T_{12}=\sqrt{\frac{z_{20}^2+2z_{10}\varepsilon_1^{-}}{\varepsilon_1^{+}(\varepsilon_1^{+}+
\varepsilon_1^{-})}}.$ Thus, the control $u_1(z)$ transfers the
point  $z_0$ to the point $z_{T_1}$ in the time
$T_1=T_{11}+T_{12}=\bigl(z_{20}+\sqrt{(z_{20}^2+2z_{10}\varepsilon_1^{-})
(\varepsilon_1^{+}+\varepsilon_1^{-})/\varepsilon_1^{+}}\bigr)
\bigl/\varepsilon_1^{-}.$

Finally, if $z_{20}=w_1(z_{10})$ and $z_{10}\le 0$ then the
control $u_1^+(z)$ transfers the point $z_0$ to the point
$z_{T_1}=\left(0,0,z_{3T_1},z_{4T_1}\right)^*$ in the time
$T_1=T_{12}.$

Thus, on the first step  the control $u_1(z)$ transfers the point
$z_0$ to the point $z_{T_1}=\left(0,0,z_{3T_1},z_{4T_1}\right)^*$
along the trajectory of the system (\ref{skorik-dnm_f1_z}) in the
time $T_1.$

On the second step  the motion continues  in the plane
$P=\{(z_1,z_2,z_3,z_4): z_1=z_2=0\}.$ To ensure this,  we choose a
control as a root of the equation
\begin{equation} \label{skorik-equation_H_1-0}
H_1(0,0,z_3,z_4,u)=\alpha u^3-u-\frac{g}{l_1}\,\sin
z_3=0.\end{equation} Notice that if $\alpha \in
(0,(4/27)l_1^2/g^2] $ then this equation has at least one positive
root and one negative root for any $z_3\in{\Bbb R}.$ Put
$$\alpha_0(z_3){=}\sqrt[3]{\frac{g\sin z_3}{2\alpha l_1}{+}\frac{1}{\alpha}\sqrt{\frac{g^2\sin^2z_3}{4l_1^2}{-}\frac{1}{27\alpha}}},
\quad \beta_0(z_3){=}\sqrt[3]{\frac{g\sin z_3}{2\alpha
l_1}{-}\frac{1}{\alpha}\sqrt{\frac{g^2\sin^2z_3}{4l_1^2}{-}\frac{1}{27\alpha}}}
.
$$
Choose $u_2^+(z_3)$ as the maximal root  and  $u_2^-(z_3)$ as the
minimal  root of the equation (\ref{skorik-equation_H_1-0}), i.e.
$$u_2^+(z_3)=\alpha_0(z_3)+\beta_0(z_3),\quad
u_2^-(z_3)=\frac{1}{2}\left(-1+i
\sqrt{3}\right)\alpha_0(z_3)-\frac{1}{2}\left(1+i
\sqrt{3}\right)\beta_0(z_3).$$  Then for all $z_3\in{\Bbb R}$ we
have
$$0<u_2^+(-\pi/2)\le u_2^+(z_3)\le u_2^+(\pi/2),\quad
u_2^-(-\pi/2)\le u_2^-(z_3)\le u_2^-(\pi/2)<0.$$  Since
$H_2(0,0,z_3,z_4,u)=u$ then
$$H_2(0,0,z_3,z_4,u_2^+(z_3))\ge\varepsilon_2^{+},\quad
H_2(0,0,z_3,z_4,u_2^-(z_3))\le -\varepsilon_2^{-},$$ where
$\varepsilon_2^{+}=u_2^+(-\pi/2),$
$\varepsilon_2^{-}=-u_2^+(\pi/2).$

Consider  the trajectories of the system (\ref{skorik-dnm_f1_z})
with the controls $u_2^\pm(z_3)$ which go to the origin, i.e. the
trajectories of the system
 $$\dot z_1=z_2,\quad\dot z_2=0,\quad \dot z_3=z_4,\quad \dot
z_4=u_2^\pm (z_3),
$$ which go to the origin. They belong to the plane $P$ and, in addition,  $z_4=w_2(z_3),$ where
$w_2(z_3)=-\sqrt{2\int\limits_0^{z_3}u^+(\zeta)d\zeta}$ if $z_3\ge
0$ and $w_2(z_3)=\sqrt{-2\int\limits_{z_3}^0u^-(\zeta)d\zeta}$ if
$z_3\le 0.$
 Define
$$ u_2(z_3,z_4)=\left\{\begin{array}{l}
u_2^+(z_3)\;\; \mbox{if} \;\;z_4<w_2(z_3)\;\;\mbox{or} \;\;
z_4=w_2(z_3)\;\;
\mbox{and}\;\;z_3\ge 0,\\
u_2^-(z_3)\;\; \mbox{if}
\;\;z_4>w_2(z_3)\;\;\mbox{or}\;\;z_4=w_2(z_3)\;\;
\mbox{and}\;\;z_3\le 0.\\[3pt]
\end{array}\right. $$
Like the first step, this control transfers the point $z_{T_1}$ to
the origin in some finite time $T_2-T_1.$

 Therefore, returning to the initial variables we have that  the  control
$$u(\varphi,\dot\varphi,\psi,\dot\psi;t)=\left\{\begin{array}{ll}
u_1(\varphi{-}\psi,\dot\varphi{-}\dot\psi,\psi,\dot\psi),& 0\le
t\le
T_1,\\
u_2(\psi,\dot\psi),& T_1\le t\le T_2,\\
\end{array}\right.$$
transfers the initial point
$(\varphi(0),\dot\varphi(0),\psi(0),\dot\psi(0))$ to the origin
along the trajectory of the system (\ref{skorik-syst_phi_psi}) in
the finite time $T=T_2.$

Let us summarize the results.  We have proved  that the stoppage
problem of a controllable two-link pendulum can be solved in the
following way. On the first step the  control is chosen so that
the angles $\varphi,$ $\psi$  and the angular speeds
$\dot\varphi,$ $\dot\psi$ become equal in the finite time moment
$T_1,$  i.e. $\varphi(T_1)=\psi_(T_1)$ and
$\dot\varphi(T_1)=\dot\psi(T_1)$ (see fig.~\ref{skorik-ndm-ris2}).
Roughly speaking, the two links of the pendulum form a one-link
pendulum of length $l=l_1+l_2.$ Further    damping of vibrations
of the two-link pendulum  preserves this configuration of the
links, i.e. we choose the control so that  $\varphi(t)=\psi(t)$
and $\dot\varphi(t)= \dot\psi(t)$ for $t\in [T_1,T_2]$ until  the
time moment $T_2$ when the stoppage occurs.

\begin{multicols}{2}
  \begin{figure}[H]
    \centering
    \includegraphics[scale=0.4]{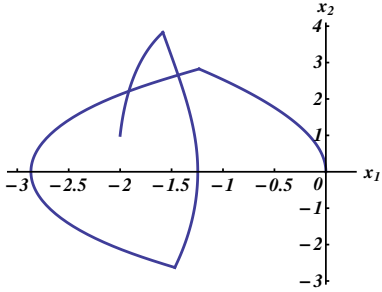}
     \vspace*{-3mm}
    \caption{Projection of the phase trajectory of the system
(\ref{skorik-dnm_f1_x}) on the plane $Ox_1x_2.$}
\label{skorik-ndm-ft1}

  \end{figure}
  \begin{figure}[H]
    \centering
    \includegraphics[scale=0.4]{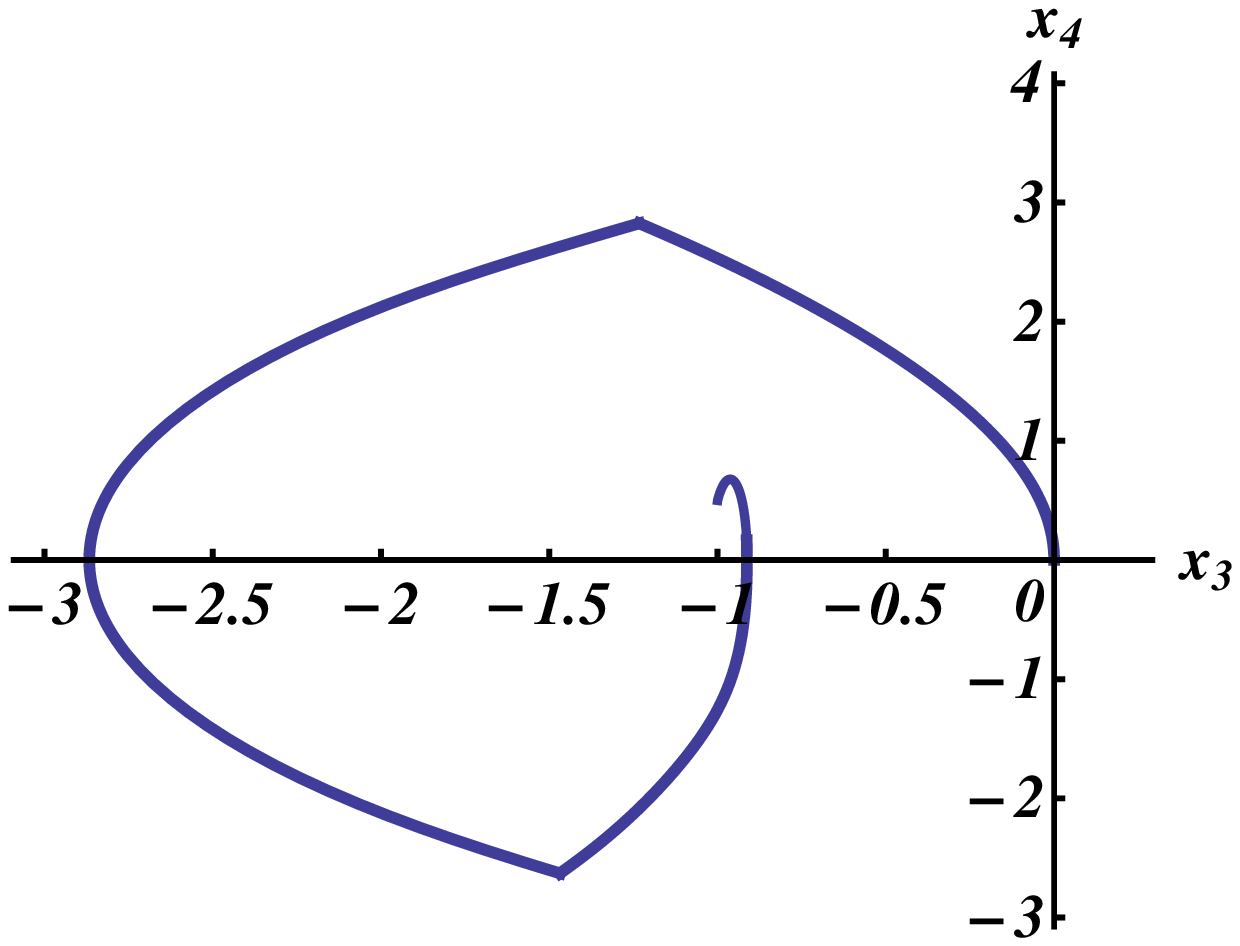}
    \vspace*{-3mm}
    \caption{Projection of the phase trajectory of the system
(\ref{skorik-dnm_f1_x}) on the plane $Ox_3x_4.$}
\label{skorik-ndm-ft2}
  \end{figure}
\end{multicols}

\begin{multicols}{2}
  \begin{figure}[H]
  \vspace*{-5mm}
    \centering
    \includegraphics[scale=0.4]{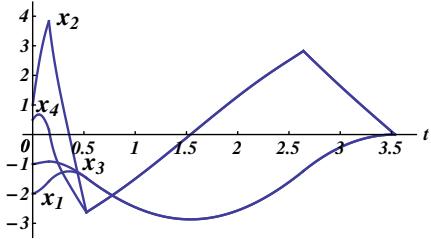}
     \vspace*{-5mm}
        \caption{Graphics of components of the trajectory $x(t).$} \label{skorik-ndm-xt}
  \end{figure}
  \begin{figure}[H]
   \vspace*{-5mm} \centering
    \includegraphics[scale=0.4]{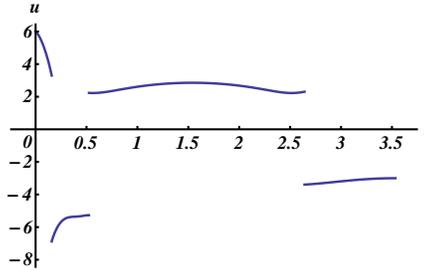}
    \vspace*{-5mm}
    \caption{Graphic of control on the trajectory.} \label{skorik-ndm-ut}
  \end{figure}
\end{multicols}

As an example, let us  transfer  the point $x_0=(-2,1,-1,0.5)^*$
to the origin  according to  the algorithm described earlier for
the system (\ref{skorik-dnm_f1_x}) with  $m_1=m_2,$ $l_1=l_2=l,$
$g/l=1,$ $\alpha=1/9$   (see
fig.~\ref{skorik-ndm-ft1}--fig.~\ref{skorik-ndm-ut}).
 In this case $\varepsilon_1^{+}=20,$
$\varepsilon_1^{-}=10$ and $T_{11}\approx 0.15814,$ $T_1\approx
0.52443,$ $T_{21}\approx 2.64102,$ $T=T_2\approx 3.53471.$

\section{Classes  of staircase systems}\label{skorik-skorik_section_7}

In this section we introduce the new classes of nonlinear systems
 which are mapped on the systems (\ref{skorik-r1_kns}). In addition,
  we  give    changes of variables    satisfying
(\ref{skorik-r1_f4})--(\ref{skorik-r1_f5}).

\subsection{\hskip-3.1mm}  Let the system  (\ref{skorik-r1_f2}) be of the form
\begin{equation} \label{skorik-r1_f20_obobshennaya_treugolnaya}
\left\{\begin{array}{l}
  \dot x_1 \qquad\;\,=  f_1(x_1,\ldots,x_{m{+}1}),\\
  \dot x_2 \qquad\;\, =   f_2(x_1,\ldots,x_{m{+}2}),\\
.\qquad.\qquad.\qquad.\qquad.\qquad.\\
 \dot x_{n{-}m}\quad = f_{n{-}m}(x_1,\ldots,x_n), \\
 \dot x_{n{-}m{+}1} = f_{n{-}m{+}1}(x_1,\ldots,x_n,u),\\
 .\qquad.\qquad.\qquad.\qquad.\qquad.\\
 \dot x_n\qquad\;\, = f_n(x_1,\ldots,x_n,u),
   \end{array}\right. 1\le m<n.  \end{equation}

 The system (\ref{skorik-r1_f20_obobshennaya_treugolnaya}) for $m=1$  was
introduced and considered in the paper \cite{korobov} and   was
named  the triangular system.

In this subsection we consider  the case $m=2$ in detail, i.e. we
consider the system
\begin{equation}  \label{skorik-r1_f20}
\left\{\begin{array}{l}
  \dot x_1 \quad =  f_1(x_1,x_2,x_3),  \\
  \dot x_2 \quad =   f_2(x_1,x_2,x_3,x_4),  \\
.\qquad.\qquad.\qquad.\qquad.\qquad.\\
 \dot x_{n-2}  = f_{n-2}(x_1,x_2,\ldots,x_n), \\
 \dot x_{n-1} = f_{n-1}(x_1,x_2,\ldots,x_n,u),   \\
 \dot x_n \quad =   f_n(x_1,x_2,\ldots,x_n,u), \qquad\qquad n \ge 3.     \\
   \end{array}\right.  \end{equation}
Here and  further $a(x)=(f_1,\ldots,f_{n-2}),$ $b_1 =(0,\ldots,0,
1, 0)^*,$ $ b_2 = (0,\ldots,0, 1)^*,$ $ \beta_1(x,u) =
f_{n-1}(x_1,\ldots,x_n,u),$ $ \beta_2(x,u)=f_n(x_1,\ldots,x_n,u).$

 Put $\varphi_1(x) = x_1,$ $\varphi_2(x) =
 x_2 $ and consider the change of variables
 \begin{equation} \label{skorik-r1_f22}
 \left\{\begin{array}{l}
 y_1= \varphi_1(x)=x_{1} \equiv F_{1}(x_{1}),\quad z_{1}=\varphi_2(x)=x_{2}\equiv \Phi_{1}(x_{1},x_{2}),\\
\displaystyle  y_k= L_a^{k-1}\varphi_1(x) =
 \sum\limits_{i=1}^{2k-3}\frac {\partial
F_{k-1}(x_{1},\ldots,x_{2k-3})}{\partial x_{i}}\ f_{i}(x_{1},\ldots,x_{i+2})\equiv \\
 \quad\;\equiv F_{k}(x_{1},\ldots,x_{2k-1}), \quad 2\le k \le p\;\; \mbox{for}\;\; n=2p
 \;\;
\mbox{or}\;\; n=2p-1,\\
\displaystyle
z_k=L_{a}^{k-1}\varphi_2(x)=\sum\limits_{i=1}^{2k-2}\frac
{\partial
\Phi_{k-1}(x_{1},\ldots,x_{2k-2})}{\partial x_{i}}f_{i}(x_{1},\ldots,x_{i+2})\equiv\\
\quad\; \equiv \Phi_{k}(x_{1},\ldots,x_{2k}),
\quad\left\{\begin{array}{lcl}
 2\le k \le p & \mbox {if} & n= 2p,\\
 2\le k \le p-1&\mbox{if} & n=2p-1.\end{array}\right.
\end{array} \right.\end{equation}
In addition, if  $ n=2p $ then put
 $$\begin{array}{l}
F_{p+1}(x,u)=\sum\limits_{i=1}^{n-2}\frac {\partial
F_{p}(x_{1},\ldots,x_{n-1})}{\partial
x_{i}}f_{i}(x_{1},\ldots,x_{i+2})  + \frac {\partial
F_p(x_1,\ldots,x_{n-1})}{\partial
x_{n-1}}f_{n-1}(x,u),\\[5pt]
 \Phi_{p+1}(x,u)=\sum\limits_{i=1}^{n-2}\frac
{\partial \Phi_{p}(x)}{\partial
x_{i}}f_{i}(x_{1},\ldots,x_{i+2}){+} \frac {\partial
\Phi_{p}(x)}{\partial x_{n-1}}f_{n-1}(x,u){+} \frac{\partial
\Phi_{p}(x)}{\partial x_{n}}f_{n}(x,u).
\end{array}$$
If $ n=2p-1 $ then  put $$\begin{array}{l} F_{p+1}(x,u)=
\sum\limits_{i=1}^{n-2}\frac {\partial F_{p}(x)}{\partial
x_{i}}f_{i}(x_{1},\ldots,x_{i+2})+ \frac {\partial
F_{p}(x)}{\partial x_{n-1}}f_{n-1}(x,u)+\frac {\partial
F_p(x)}{\partial
x_n}f_n(x,u),\\[5pt]
\Phi_p(x,u)=\sum\limits_{i=1}^{n-2}\frac {\partial
\Phi_{p-1}(x_1,\ldots,x_{n-1})}{\partial
x_i}f_i(x_1,\ldots,x_{i+2}){+} \frac {\partial
\Phi_{p-1}(x_1,\ldots,x_{n-1})}{\partial
x_{n-1}}f_{n-1}(x,u).\end{array}
$$

For solvability of the system (\ref{skorik-r1_f22}) with respect
to $x_1,$ $\ldots,$ $x_n$  we require  that $\left|\frac{\partial
f_i(x_1,\ldots,x_{i+2})}{\partial x_{i+2}}\right| \ge a>0$ for
$i=1,\ldots,n-2.$ Analogously to the  paper \cite{korobov}, we
prove   the equalities
$$\begin{array}{l}
\displaystyle \frac{\partial F_k(x_1,\ldots,x_{2k-1})}{\partial
x_{2k-1}} =\prod\limits_{i=1}^{k-1} \frac{\partial
f_{2i-1}(x_1,\ldots,x_{2i+1})}{\partial x_{2i+1}},\quad
k=2,\ldots, p,  \\
 \displaystyle \frac{\partial
\Phi_k(x_1,\ldots,x_{2k})}{\partial x_{2k}} =
\prod\limits_{i=1}^{k-1} \frac{\partial
f_{2i}(x_1,\ldots,x_{2i+2})}{\partial x_{2i+2}},\quad
k=\left\{\!\!\begin{array}{lcl} 2,\ldots, p{-}1&\mbox{if}& n=2p{-}1,\\
2,\ldots, p&\mbox{if}& n=2p.\end{array}\right. \end{array}$$
Hence,
$$  \left|\frac{\partial F_k(x_1,\ldots,x_{2k-1})}{\partial
x_{2k-1}}\right|\ge a^{k-1}>0, \quad
 \left|\frac{\partial
\Phi_k(x_1,\ldots,x_{2k})}{\partial x_{2k}}\right| \ge a^{k-1}>0.
$$ Thus, the change of variables (\ref{skorik-r1_f22}) is nonsingular.

Let us explain how to solve  the system (\ref{skorik-r1_f22}) with
respect to $x_1,$ $\ldots,$ $x_n.$  At the beginning we have $
x_1=y_1= h_1(y_1),$ $x_2=z_1=h_2(z_1),$ according to the change of
variables. Suppose that for certain $k\ge 2$ the variables $x_1,$
$\ldots,$ $x_{2k-2}$ are found and have the form
$$\begin{array}{l}
x_3=h_3(y_1,y_2,z_1),\\
x_4=h_4(y_1,y_2,z_1,z_2),\\
.\qquad.\qquad.\qquad.\qquad.\\
 x_{2k-3} = h_{2k-3}(y_1,\ldots,y_{k-1},z_1,\ldots,z_{k-2}),\\
  x_{2k-2} = h_{2k-2}(y_1,\ldots,y_{k-1}, z_1,\ldots,z_{k-1}).
  \end{array}$$
Consider the functions $\widehat F(x_{2k-1})=
F_{k}(x_1,\ldots,x_{2k-2},x_{2k-1}),$ $\widehat\Phi(x_{2k})=
\Phi_{k}(x_1,\ldots,$ $x_{2k-1},x_{2k}).$ The  functions $\widehat
F(x_{2k-1}),$ $\widehat\Phi(x_{2k})$   are  one-to-one mappings of
$\Bbb R$ to $\Bbb R$. From
 the
equation
$$\begin{array}{c}
 y_k = F_k(x_1,\ldots,x_{2k-1}) =F_k\Bigl(h_1(y_1),h_2(z_1),\ldots,
 h_{2k-3}(y_1,\ldots\\
 \ldots,y_{k-1},z_1,\ldots,z_{k-2}), h_{2k-2}(y_1,\ldots,
 y_{k-1},z_1,\ldots,z_{k-1}),x_{2k-1}\Bigr),\end{array}$$
we find  $x_{2k-1}=h_{2k-1}(y_1,\ldots,y_k,z_1,\ldots,z_{k-1}).$
Substituting this expression to the equation
 $z_k   =
\Phi_k(h_1(y_1),h_2(z_1),\ldots,
h_{2k-2}(y_1,\ldots,y_{k-1},\ldots,z_1,\ldots,z_{k-1}), x_{2k-1},
x_{2k})$ we obtain
 $$\begin{array}{c} z_k   =\Phi_k(x_1,\ldots,x_{2k}) =
\Phi_k(h_1(y_1),h_2(z_1),\ldots, h_{2k-2}(y_1,\ldots\\
\ldots,y_{k-1},\ldots,z_1,\ldots,z_{k-1}),h_{2k-1}(y_1,\ldots,y_k,z_1,\ldots,z_{k-1}),
x_{2k}).\end{array}$$ From this equation we find   $x_{2k}=
h_{2k}(y_1,\ldots,y_k,z_1,\ldots,z_k).$

Thus, if  $n=2p$ then the nonsingular change of variables
(\ref{skorik-r1_f22}) maps  the system (\ref{skorik-r1_f20})  to
the system
\begin{equation} \label{skorik-chetnyi_slychai}
\left\{\begin{array}{l} \dot y_i=y_{i{+}1},\; i=1,\ldots,p{-}1,\;
\dot y_p=F_{p{+}1}\bigl(h_1(y_1),h_2(z_1),...,
 h_{2p-1}(y_1,...,y_p,\\
\qquad z_1,...,z_{p-1}), h_{2p}(y_1,\ldots,
 y_{p},z_1,...,z_p),u\bigr)\equiv H_1(y_1,...,y_p,z_1,...,z_p,u),\\
\dot z_i=z_{i{+}1},\; i=1,\ldots,p{-}1,\; \dot
z_p=\Phi_{p{+}1}(h_1(y_1),h_2(z_1),...,
h_{2p-1}(y_1,...,y_p,\\
\qquad z_1,...,z_{p-1}),h_{2p}(y_1,...,y_p,z_1,...,z_p),u)\equiv
H_2(y_1,...,y_p,z_1,...,z_p,u),
\end{array}\right.\end{equation}
and if $n=2p-1$ then one maps the system (\ref{skorik-r1_f20}) to
the system
\begin{equation} \label{skorik-kanonical_system_staicase}
\left\{\begin{array}{l} \dot y_i=y_{i+1},\quad
i=1,\ldots,p-1,\qquad
\dot y_p\quad=H_1(y_1,\ldots,y_p,z_1,\ldots,z_{p-1},u),\\
\dot z_i=z_{i+1},\quad i=1,\ldots,p-2,\qquad \dot
z_{p-1}=H_2(y_1,\ldots,y_p,z_1,\ldots,z_{p-1},u).
\end{array}\right.\end{equation}

In the partial case when the first $(n-2)$ equations of the system
(\ref{skorik-r1_f20})  are linear with respect to  the last
argument, i.e. the system has the form
$$\left\{ \begin{array}{lll}
\displaystyle   \dot x_i \quad\, =
f_i(x_1,\ldots,x_{i+1})+c_ix_{i+2},
   \qquad  i=1,\ldots,n{-}2,  \\
 \dot x_{n-1}= f_{n-1}(x_1,\ldots,x_n,u),   \quad
 \dot x_n =   f_n(x_1,\ldots,x_n,u),      \\
   \end{array}\right.\quad  \prod\limits_{i=1}^{n-2}c_i\ne 0, $$
 the  system (\ref{skorik-r1_f22})  is solvable  with respect to
$x_1,\ldots,x_n$ in an obvious way analogously to
\cite{skorik-sklyar_e_v}. For example, the nonsingular change of
variables $y_1=x_1,$ $y_2=x_3,$ $z_1=x_2,$ $z_2=x_1^2+x_4$ maps
the system
$$
\dot x_1=x_3,\,\dot x_2=x_1^2{+}x_4,\, \dot
x_3=f_1(x_1,x_2,x_3,x_4)\cos u,\, \dot
x_4=f_2(x_1,x_2,x_3,x_4)\sin u{-}2x_1x_3$$ to the system
$$\dot y_1=y_2,\;
\dot y_2=f_1(y_1,z_1,y_2,z_2-y_1^2)\cos u,\; \dot z_1=z_2,\; \dot
z_2=f_2(y_1,z_1,y_2,z_2-y_1^2)\sin u.$$

\subsection{\hskip-3.1mm}  For $n=2p-1\ge 2$  we  consider  the system
\begin{equation}  \label{skorik-r1_f23}
\left\{ \begin{array}{l}
  \quad\;\dot x_1\quad \; =  f_1(x_1,x_2,x_3),  \\
  \left\{\begin{array}{l}\dot x_{2i}\quad =   f_{2i}(x_1,\ldots,x_{2i+3}),  \\
 \dot x_{2i+1}= f_{2i+1}(x_1,\ldots,x_{2i+3}), \end{array}\right.\quad i=1,\ldots,p-2, \\
\quad\; \dot x_{n-1} \, = f_{n-1}(x_1,\ldots,x_n,u), \\\quad\;
\dot x_n \quad \, = f_n(x_1,\ldots,x_n,u).
\end{array}\right. \end{equation}

 Put $\varphi_1(x_1,\ldots,x_n)=x_1,$
 $\varphi_2(x_1,\ldots,x_n)=x_2.$ Then
\begin{equation}  \label{skorik-r1_f30}
\left\{\!\!\begin{array}{l}
y_1=x_1\equiv F_1(x_1), \\
 y_k= L_{a}^{k-1}\varphi_1
=\frac{d}{dt}F_{k-1}(x_{1},\ldots,x_{2k-3}) \equiv
F_k(x_1,\ldots,x_{2k-1}), \;
2\le k\le p,\\
z_1=x_2\equiv \Phi_1(x_1,x_2),\\
 z_k= L_{a}^{k-1}\varphi_2 =
 \frac{d}{dt}
\Phi_{k-1}(x_{1},\ldots,x_{2k-1})\equiv
\Phi_k(x_1,\ldots,x_{2k{+}1}),\;  2\le k\le
p{-}1,\end{array}\right. \end{equation}
 and  $ F_{p+1}(x,u)=\frac{d}{dt} F_p(x_1,\ldots,x_n),$ $
\Phi_p(x,u)=\frac{d}{dt} \Phi_{p-1}(x_1,\ldots,x_n). $ Following
the paper \cite{kou_elliot_tarn}, suppose that $$
\Delta_0=\left|\frac{\partial f_1}{\partial x_3}\right|\ge
\varepsilon_0,\quad  \Delta_i= \left|
 \frac{\partial f_{2i}}{\partial
x_{2i+2}}\frac{\partial f_{2i+1}}{\partial x_{2i+3}}-
\frac{\partial f_{2i}}{\partial x_{2i+3}} \frac{\partial
f_{2i+1}}{\partial x_{2i+2}} \right|\ge \varepsilon_i,\;
i=1,\ldots,p-2.$$ Then the change of variables
(\ref{skorik-r1_f30}) is nonsingular and maps the system
(\ref{skorik-r1_f23}) to the system
(\ref{skorik-kanonical_system_staicase}).

\subsection{\hskip-3.1mm} For $n=2p$ we  consider the system
\begin{equation} \label{skorik-r1_f41}
\left\{
\begin{array}{l}\!\!\!\left\{
\begin{array}{l}
 \dot x_{2i-1}  =  f_{2i-1}(x_1,\ldots,x_{2i+2}),  \\
 \dot x_{2i}\quad  =  f_{2i}(x_1,\ldots,x_{2i+2}),
 \end{array}\right.\quad   i=1,\ldots, p-1,\\
\;\;\,\dot x_{n-1}\,= f_{n-1}(x_{1},\ldots,x_n,u), \\
\;\;\,  \dot x_n\quad \,=  f_n(x_1,\ldots,x_n,u). \\
 \end{array}\right.  \end{equation}

Put $\varphi_1(x)=x_1,$ $\varphi_2(x)=x_2$ and
\begin{equation}  \label{skorik-r1_f42}
  y_k= L_{a}^{k-1}\varphi_1(x),\quad
 z_k= L^{k-1}_a\varphi_2(x),\quad  k=1,\ldots,p.
\end{equation}

 Suppose that
$$
\left|\! \displaystyle \frac{\partial
f_{2i{-}3}(x_1,\ldots,x_{2i})}{\partial x_{2i-1}}\displaystyle
\frac{\partial f_{2i{-}2}(x_1,\ldots,x_{2i})}{\partial x_{2i}}-
\displaystyle \frac{\partial
f_{2i{-}3}(x_1,\ldots,x_{2i})}{\partial x_{2i}}\displaystyle
\frac{\partial f_{2i{-}2}(x_1,\ldots,x_{2i})}{\partial x_{2i-1}}\!
\right|\ge \varepsilon $$  for $i=2,\ldots,p$ and $\varepsilon >
0.$
  Then the change of variables
(\ref{skorik-r1_f42}) is  nonsingular  and  maps the system
(\ref{skorik-r1_f41}) to  the system
 (\ref{skorik-chetnyi_slychai}).

\subsection{\hskip-3.1mm}  For a fixed $k$ such that $8\le 2k\le n{+}1$ we  consider the system
\begin{equation}  \label{skorik-r44_f1}
\left\{\begin{array}{l}
 \dot x_i\quad\,  =  f_i(x_{1},\ldots,x_{k}), \qquad\;\;\; i=1,\ldots,k{-}1,   \\
 \dot x_i\quad\,  =  f _i(x_{1},\ldots,x_{i+1}), \qquad  i=k,\ldots,n{-}3, \\
\dot x_{n-2}= f_{n-2}(x_{1},\ldots,x_{n}),\\
 \dot x_{n-1} =  f_{n-1}(x_{1},\ldots,x_{n},u), \\
 \dot x_n \quad =  f_{n}(x_{1},\ldots,x_{n},u). \\
 \end{array}\right.\end{equation}

  Put $\varphi_1(x)=x_1$ and $\varphi_2(x)=x_{n-k+1}.$
Then the  change of variables
$$ y_s=L^{s-1}_a
\varphi_1(x),\;\; s=2,\ldots,n{-}k{+}1,\quad  z_s=L^{s-1}_a
\varphi_2(x),\;\; s=2,\ldots,k{-}1,$$ maps the system
(\ref{skorik-r44_f1}) to the system
$$\left\{\begin{array}{lr}
\dot y_i=y_{i+1},\quad i=1,\ldots,n{-}k,&
\dot y_{n-k+1}=H_1(y_1,\ldots,y_{n-k+1},z_1,\ldots,z_{k-1},u),\\
\dot z_i=z_{i+1},\quad i=1,\ldots,k{-}2,& \dot
z_{k-1}=H_2(y_1,\ldots,y_{n-k+1},z_1,\ldots,z_{k-1},u).
\end{array}\right.$$

\end{document}